\documentclass[12pt]{article}
\usepackage{fullpage}
\usepackage{latexsym,amsfonts,amssymb,amsthm,amsmath}
\usepackage{enumerate}

\usepackage{booktabs} 
\usepackage{todonotes}
\usepackage[normalem]{ulem}
\usepackage[utf8]{inputenc}
\usepackage{xspace}
\usepackage{bm}
\usepackage{comment}
\usepackage{multicol}

\numberwithin{equation}{section}

\newcommand{\R}{\ensuremath{\mathbb{R}}\xspace}
\newcommand{\Rn}[1][3]{\ensuremath{\R^#1}\xspace}

\newcommand{\bs}[1]{\boldsymbol{#1}}
\newcommand{\n}{\ensuremath{\bs{n}}\xspace}

\newcommand{\Dg}{D_{\Gamma}}
\newcommand{\gradg}{\nabla_{\Gamma}}

\DeclareMathOperator*{\argmin}{arg\, min}
\DeclareMathOperator{\Lin}{Lin} 
\DeclareMathOperator{\Div}{div} 
\DeclareMathOperator{\Divg}{div_\Gamma}  
\newcommand{\lb}{\Delta_{\Gamma}}

\newcommand{\tensor}{\otimes}

\newcommand{\definedas}{:=}

\DeclareMathOperator{\tr}{tr}

\newcommand{\Ad}{\mathcal{A}}
\newcommand{\Vel}{\mathcal{V}}
\newcommand{\D}{\mathbf{D}}
\newcommand{\V}{\mathbb{V}}

\usepackage{authblk}
\theoremstyle{theorem}
\newtheorem{thrm}{Theorem}
\newtheorem{lmm}[thrm]{Lemma}
\newtheorem{crllr}[thrm]{Corollary}
\newtheorem{prpstn}[thrm]{Proposition}

\theoremstyle{remark}
\newtheorem{rmrk}[thrm]{Remark}
\newtheorem{dfntn}[thrm]{Definition}

\begin{document}

\title{The shape derivative of the Gauss curvature%
\thanks{Partially supported by CONICET through grants PIP 112-2011-0100742 and 112-2015-0100661, by Universidad Nacional
	del Litoral through grants CAI+D 501 201101 00476 LI, by Agencia Nacional de
	Promoci\'on Cient\'ifica y Tecnol\'ogica, through grants PICT-2012-2590 and PICT-2014-2522 (Argentina)}}
\author{An\'{i}bal Chicco-Ruiz, 
	Pedro Morin, and
	M.~Sebastian Pauletti\\
	UNL, Consejo Nacional de Investigaciones Cient\'ificas y T\'ecnicas, FIQ,
	Santiago del Estero 2829, S3000AOM, Santa Fe, Argentina.\\
	\texttt{achicco,pmorin,pauletti@santafe-conicet.gov.ar}
}

\maketitle

\begin{abstract}
  We introduce new results about the shape derivatives of 
  scalar- and vector-valued functions. 
  They extend the results from~\cite{DN2012} 
  to more general surface energies. In~\cite{DN2012}  Do\u{g}an and Nochetto consider 
  surface energies defined as integrals over surfaces of functions 
  that can depend on the position, the unit normal and the mean curvature
  of the surface.
  In this work we present a systematic way to derive
  formulas for the shape derivative of more general geometric quantities, including the Gauss curvature 
  (a new result not available in the literature)
  and other geometric invariants
  (eigenvalues of the second fundamental form).
  This is done for hyper-surfaces in the Euclidean space of any 
  finite dimension.
  As an application of the results, with relevance for numerical
  methods in applied problems, we introduce a new scheme
  of Newton-type to approximate a minimizer of a shape functional.
  It is a mathematically sound generalization of the method presented in
  \cite{CMP2016}. 
  We finally find the particular formulas for the first and second order shape derivative of the area and the Willmore functional,
  which are necessary for the Newton-type method mentioned above.
  
  
  
\end{abstract}

\section{Introduction}
Energies that depend on the domain appear in applications in many areas, from 
materials science, to biology, to image processing. 
Examples when the domain dependence of the energy occurs through surfaces include the minimal surface problem, the study of the shape of droplets (surface tension), image segmentation and shape of biomembranes, to name a few.
In the language of the shape derivative theory
\cite{SZ1992,DZ2011,Wal2015}, these energies are called shape functionals.
This theory provides a solid mathematical framework to pose and solve minimization problems for such functionals. 

For most of the problems of interest, the energy (shape functional) can be cast as $\int_{\Gamma} F(\text{``geometrical quantities''})$,
where ``geometrical quantities'' stands for quantities such as the normal $\n$, 
the mean curvature $\kappa$, the Gauss curvature $\kappa_g$, or in general any quantity that is well defined 
for a surface $\Gamma$ as a geometric object, i.e.,  independent of the 
parametrization.
For example, $F=1$ in the case of minimal surface,
$F=F(x,\n)$ is used in the modeling of crystals \cite{ATW1993,AT1995,Tay1978,Tay1992}
in materials science. 
The Willmore functional corresponds to $F=\frac{1}{2}\kappa^2$\cite{Wil982}---where $\kappa$ is
the mean curvature---and the related spontaneous curvature functional to $F=\frac{1}{2}(\kappa-\kappa_0)^2$; they are used in models for the bending energy of membranes, particularly in the study of biological vesicles \cite{Hel1973,Jen1977,Jen1977b,SBL1991}.
The modified form of the Willmore functional, which corresponds to $F=g(x)\kappa^2$,
is applied to model biomembranes when
the concentration or composition of lipids changes spatially \cite{BDWJ2005,CKV2007}.

The minimization of these energies requires the knowledge of their
(shape) derivatives with respect to the domain
and has motivated researchers to seek
formulas for the shape derivative of the normal and the mean
curvature. 
The shape derivative of the normal is simple and can be found
in \cite{DZ2011,Wal2015} among other references.
Particular cases of $F=F(x,n)$ are derived in \cite{BP1996, MWBC.EA1993, Tay1992}.
The shape derivative of the mean curvature or particular cases of $F=F(\kappa)$ can also be 
found in \cite{Wil1993,HR2003,Rus2005, Dzi2008, DB2010, DN2012, Wal2015},
where the shape derivative is computed from scratch; some using parametrizations, 
others in a more coordinate-free setting using the signed distance function,
but in general the same computations are repeated each time a new functional 
dependent on the mean curvature appears.
A more systematic approach to the computations is found in \cite{DN2012}, where
Do\u{g}an and Nochetto propose a formula for the shape derivative of a functional of the form
$F=F(x,\n,\kappa)$, that relies on knowing the shape derivatives of $\n$ and $\kappa$.
They rightfully assert that by having this formula at hand, it wouldn't be necessary to 
redo all the computations every time a new functional depending on these quantities appears.

The main motivation of this article is to find such a formula when $F$ also 
depends on the Gauss curvature $\kappa_g$
which, as far as we know, has not been provided elsewhere.
In fact, we let $F$ also depend on differential operators
of basic geometric quantities such as $\nabla_\Gamma\kappa $ or $\Delta_\Gamma\kappa$.
These are important when second order shape derivatives are necessary in Newton-type methods for minimizing functionals.
%


Our new results (Section \ref{sec:main}) allow us to develop a more systematic approach 
to compute shape derivatives of integrands that are functional relations of geometric quantities.
The method, starting from the shape derivative of the normal, provides a
formula for the shape derivative of higher order tangential derivatives of 
geometrical quantities.
In particular we give a nice formula for the shape derivative of the gaussian 
curvature
and extend the results of \cite{DN2012} to more generals integrands.

These results are also instrumental to develop a relevant numerical
method in applied problems. More precisely, we introduce a new scheme
of the Newton-type to find a minimizer of a surface shape functional.
It is a mathematically sound generalization of the method used in
\cite{CMP2016}. 

In Section~\ref{S:prelim} we state some preliminary concepts and elements of 
basic tangential calculus.
In Section~\ref{S:functionals} we recall the concept of shape differentiable 
functionals through the velocity method.
In Section~\ref{S:shape-domain} we motivate and introduce the concept of shape 
derivative of functions involved in the definition of shape functionals through 
integrals over the domain.
In Section~\ref{S:shape-boundary} we motivate and introduce the concept of 
shape derivative of functions involved in the definition of shape functionals 
through integrals over the boundaries of domains.
In Section~\ref{S:domain-properties} we explore the relationship between the 
shape derivative of domain functions and the classical derivative operators.
In Section~\ref{S:boundary-properties} we explore the relationship between the 
shape derivative of boundary functions and the tangential derivative operators.
These last sections set the foundations for Section~\ref{sec:main} where the 
shape derivatives of the tangential derivatives of geometric quantities are 
obtained.
We end with Sections~\ref{sec:invariants} and~\ref{s:newton} where we 
apply the results to obtain the shape derivatives of the Gauss curvature, the 
geometric invariants and introduce a quasi Newton method in the language of 
shape derivatives whose formula is then computed for the Willmore 
functional.

\section{Preliminaries}\label{S:prelim}
\subsection{General concepts}
Our notation follows closely that of~\cite[Ch.~2, Sec.~3]{DZ2011}. A 
\emph{domain} is an open and bounded subset $\Omega$ of $\Rn[N]$, and a 
\emph{boundary} is the boundary of some domain, i.e., $\Gamma=\partial \Omega$.
An \emph{$N-1$ dimensional surface} in $\Rn[N]$ can be thought of as a 
reasonable subset of a boundary in $\Rn[N]$.
If a boundary $\Gamma$ is smooth, we denote the \emph{normal vector field} 
by $\bs{n}$ and assume that it points outward of $\Omega$. 
The \emph{principal curvatures}, denoted by $\kappa_1, \dots, \kappa_{N-1}$, 
are the eigenvalues of the second fundamental form of $\Gamma$, which are all 
real.
The \emph{mean curvature} $\kappa$ and \emph{Gaussian curvature} $\kappa_g$ are
\begin{equation}\label{curvatures}
\kappa = \sum_{i=1}^{N-1} \kappa_i \qquad \text{ and }\qquad 
\kappa_g =\prod_{i=1}^{N-1} \kappa_i. 
\end{equation} 
We will obtain analogous and more useful definitions of 
$\kappa$ and $\kappa_g$ using tangential derivatives of the normal; see 
(\ref{mcurvature2}) and (\ref{gcurvature2}) for $N=3$, and Section 
\ref{sec:invariants} for any dimension, where we introduce the geometric 
invariants of a surface, following Definition 3.46 of \cite{Kuhnel2013}.
  
In the scope of this work a \emph{tensor} $S$ is a bounded, linear operator 
from a normed vector space $\mathbb{V}$ to itself. 
The set of tensors is denoted by $\Lin(\mathbb{V})$.
If $ \dim(\mathbb{V})=N$, $S$ can be represented by an $N\times N$ matrix 
$S_{ij}$. We will mainly consider $\mathbb{V}=\Rn[N]$.
   
For a vector space with a scalar product, the \emph{tensor product}  of two 
vectors $\bs{u}$ and $\bs{v}$ is the tensor $\bs{u}\tensor \bs{v}$ which 
satisfies  $(\bs{u}\tensor \bs{v})\bs{w} =(\bs{v}\cdot \bs{w})\bs{u}$. The 
\emph{trace} of a tensor $S$ is $\tr (S) = \sum_i S\bs{e}_i \cdot \bs{e}_i$, 
with $\{ \bs{e}_i\}$ any orthonormal basis of $\mathbb{V}$. The trace of a 
tensor $\bs{u}\tensor \bs{v}$ is $\tr(\bs{u}\tensor \bs{v})=\bs{u}\cdot\bs{v}$. 
 
The \emph{scalar product of tensors} $S $ and $T$ is given by $S:T=\tr(S^T T)$, 
where $S^T$ is the transpose of $S$, which satisfies $S\bs{u}\cdot 
\bs{v}=\bs{u}\cdot S^T\bs{v}$, and the \emph{tensor norm} is $|S|=\sqrt{S:S}$. 
From~\cite[Ch.~I]{Gur1981} we have the following properties:

\begin{lmm}[Tensor Properties]\label{le:tensor}
For vectors $\bs{u}$, $\bs{v}$, $\bs{a}$,  $\bs{b} \in \mathbb{V}$, and tensors 
$S$, $T$, $P \in \Lin(\mathbb{V})$, we have:
\begin{itemize}
\item $S(\bs{u}\tensor\bs{v})=S\bs{u}\tensor\bs{v}$ and 
  $(\bs{u}\tensor\bs{v})S=\bs{u}\tensor S^T\bs{v}$,
\item $I:S=tr(S)$,
\item $ST:P=S:PT^T = T:S^TP$,
\item $S: \bs{u}\tensor\bs{v}=u\cdot S\bs{v}$,
\item$(\bs{a}\tensor\bs{b}):(\bs{u}\tensor\bs{v})= 
  (\bs{a}\cdot\bs{u})(\bs{b}\cdot\bs{v})$,
\item $S:T=S:T^T = \frac12 S:(T+T^T)$ if $S$ is symmetric.
\end{itemize}
\end{lmm}

\subsection{The signed distance function}\label{S:signed distance function}
For a given domain $\Omega \subset \Rn[N]$, the \emph{signed distance function} 
$b=b(\Omega):\Rn[N]\rightarrow \R$ is given by
$b(\Omega)(x)=d_\Omega(x)-d_{\Rn[N]\setminus\Omega}(x)$, where 
$d_\Omega(x)=\inf_{y\in \Omega}|y-x|$.
If $\Omega$ is of class $\mathcal{C}^{1,1}$ (see Definitions 3.1 and 3.2 in 
\cite[Ch.~2]{DZ2011}) then for each $x\in \Gamma=\partial\Omega$ there exists a 
neighborhood $W(x)$ such that  $b\in \mathcal{C}^{1,1}(\overline{W(x)})$, 
$|\nabla b |^2 = 1$ in $W(x)$ and $\nabla b = \bs{n}\circ p$, where $\bs{n}$ is 
the unit normal vector field of $\Gamma$ and $p=p_\Gamma$ is the projection onto 
$\Gamma$ which is well defined for $y \in W(x)$ as $p(y) = 
\arg\min_{z\in\Gamma}|z-y|$  (see Theorem 8.5 of \cite[Ch.~7]{DZ2011}).
    
Moreover, if $\Omega$ is a $\mathcal{C}^2$ domain with compact boundary $\Gamma$ 
then there exists a tubular neighborhood $S_h(\Gamma)$ such that $b\in 
\mathcal{C}^{2}(S_h(\Gamma))$ (\cite[Ch.~9,p.~492]{DZ2011}), and 
$\Gamma$ is a $\mathcal{C}^2$-manifold of dimension $N-1$. Therefore,  
$\nabla b$ is a $\mathcal{C}^{1}$ extension for $\bs{n}(\Gamma)$ which 
satisfies
\begin{equation}\label{eq:eikonal}
|\nabla b|^2 \equiv 1 \text{ in } S_h(\Gamma).
\end{equation}
This Eikonal equation readily implies
\begin{equation}\label{eq:eikonal2}
D^2b\, \nabla b \equiv 0. 
\end{equation}
Also, if $\Omega$ is $\mathcal{C}^3$, we can differentiate \eqref{eq:eikonal2} 
to obtain 
\begin{equation}\label{eq:eikonal2.5}
\Div (D^2b)\cdot \nabla b =-|D^2b|^2 
\end{equation}
where we have used the product rule formula  $\Div(S^T\bs{v})=S: 
\nabla\bs{v}+\bs{v}\cdot \Div S$ where $S$ and $\bs{v}$ are tensor and vector 
valued differentiable functions, respectively,  with $S=D^2b$ and $\bs{v}=\nabla 
b$ (see \cite{Gur1981}, page 30). The divergence $\Div S$ of a tensor valued 
function is a vector which satisfies $\Div S \cdot \bs{e} =\Div (S^T\bs{e}) $ 
for any vector $\bs{e}$.

Applying the well known identity~\cite[p. 32]{Gur1981}
\begin{equation}\label{eq:div-formula}
\Div(D\bs{v}^T)=\nabla(\Div \bs{v}),
\end{equation}
to $\bs{v}=\nabla b$ we can write (\ref{eq:eikonal2.5}) as follows:       
\begin{equation}\label{eq:eikonal3}
\nabla \Delta b \cdot \nabla b =-|D^2b|^2. 
\end{equation} 
  
Since $\bs{n}(\Gamma)= \nabla b |_\Gamma$, we can obtain from $b$ more geometric 
information about $\Gamma$. Indeed, the $N$ eigenvalues of $D^2 b|_\Gamma$ are the
principal curvatures $\kappa_1, \kappa_2,\dots,\kappa_{N-1}$ of $\Gamma$ and zero~\cite[Ch.~9, p.~500]{DZ2011}.
The mean curvature of $\Gamma$, given by (\ref{curvatures}), can also be obtained as 
$\kappa =\tr D^2b = \Delta b$ (on $\Gamma$).
Also, $|D^2b|^2=\tr (D^2b)^2= \sum \kappa_i^2$, the sum of the square of the 
principal curvatures so that the Gaussian curvature is $\kappa_g 
=\frac{1}{2}\left[(\Delta b)^2 - |D^2b|^2\right]$;
notice that the right-hand side of this last identity makes sense in $S_h(\Gamma)$ whereas the left-hand side is defined only on $\Gamma$, so that the equality holds on $\Gamma$. Moreover, from (\ref{eq:eikonal3})  we obtain that 
\begin{equation*} \frac{\partial \Delta b }{\partial n} = - \sum \kappa_i^2, 
\end{equation*} and with a slight abuse of notation we may say that 
$\frac{\partial \kappa}{\partial n} = - \sum \kappa_i^2$.

The \emph{projection} of a point $x \in S_h(\Gamma)$ onto $\Gamma$ is given by  
\cite[Ch.~9, p.~492]{DZ2011} \begin{equation}\label{eq:projection} 
p(x)=x-b(x)\nabla b(x), \end{equation} and also, for any $x\in S_h(\Gamma)$, 
\emph{the orthogonal projection operator} of a vector of $\Rn[N]$ onto the 
tangent plane $T_{p(x)}(\Gamma)$, for $\Gamma \in \mathcal{C}^1$, is given by 
$P(x)=I-\nabla b(x) \otimes \nabla b(x)$. Note that the tensor $P(x)$ is 
symmetric and 
\begin{equation}\label{eq:vector-projection}
P=I-\bs{n}\tensor \bs{n} \qquad \text{on }\Gamma.
\end{equation}

It will be useful to know that the Jacobian of the projection vector field 
$p(x)$ is given, for $\Gamma \in \mathcal{C}^2$, by
\begin{equation}\label{eq:projection-derivative}
D p(x)=P(x) - b(x) D^2b(x)
\end{equation}
and satisfies $Dp|_\Gamma = P$ because $b=0$ on $\Gamma$.

\subsection{Elements of tangential calculus}\label{sec:tangential}
Following \cite[Ch.~9, Sec.~5]{DZ2011} we will introduce some basic elements of 
differential calculus on a $\mathcal{C}^1$-submanifold of codimension 1 denoted 
by $\Gamma$. This approach avoids local bases and coordinates by using intrinsic 
tangential derivatives. 
All proofs can be found in the cited book, except for Lemmas 
\ref{le:DgProperties} and \ref{le:div-formula-tan}, which are proved below.
\begin{dfntn}[Tangential Derivatives]\label{def:tangential-derivative}
Assume that $\Gamma\subset\partial\Omega$ 
and there exists a tubular neighborhood $S_h(\Gamma)$ such that  $b=b(\Omega)
\in  \mathcal{C}^1(S_h(\Gamma))$. For a scalar field $f\in 
\mathcal{C}^1(\Gamma)$ and a vector field $\bs{w}\in\mathcal{C}^1(\Gamma, 
\Rn[N])$ we define the \emph{tangential derivative operators} as
\begin{gather*}
\gradg f \definedas 
(I-\n\tensor \n) \nabla F; \quad
\Dg \bs{w}   \definedas D \bs{W} - D \bs{W}\bs{n}\otimes\bs{n}; 
\quad
\Divg \bs{w} \definedas\Div \bs{W} - D \bs{W}\bs{n}\cdot\bs{n}, 
\end{gather*}
where $F$ and $\bs{W}$ are $\mathcal{C}^1$-extensions to a neighborhood of 
$\Gamma$ of the functions $f$ and $\bs{w}$, respectively.

For a scalar function $f\in\mathcal{C}^2(\Gamma)$, 
the second order tangential derivative is given by $\Dg^2 f = \Dg (\gradg f)$, 
which is not a symmetric tensor, and
  the \emph{Laplace-Beltrami operator} (or tangential laplacian) is given by 
$\lb f = \Divg \gradg f$.

\end{dfntn}

Using the orthogonal projection operator $P$ given by 
(\ref{eq:vector-projection}), we can write 
\begin{equation*}
\gradg f =(P\,\nabla F)|_\Gamma,\qquad   
\Dg \bs{w}=(D\bs{W}\,P)|_\Gamma, \qquad
\Divg \bs{w} =(P:D\bs{W})|_\Gamma.
\end{equation*}

As it was proved in the cited book \cite{DZ2011} these definitions are 
intrinsic, that is, they do not depend on the chosen extensions of $f$ and 
$\bs{w}$ outside $\Gamma$.
Among all extensions of $f$, there is one, for $\Gamma \in \mathcal{C}^2$, that 
simplifies the calculation of $\gradg f$. That extension is $f\circ p$, where 
$p$ is the projection given by (\ref{eq:projection}), and we call it the 
\emph{canonical extension}. 
The following properties of the canonical extensions are proved 
in~\cite[Ch.~9, Sec.~5.1]{DZ2011} 
\begin{lmm}[Canonical extension]\label{le:canonical-ext}
For $\Gamma$, $f$ and $\bs{w}$ satisfying the assumptions of 
Definition \ref{def:tangential-derivative}, consider  $F=f\circ p$, and 
$\bs{W}=\bs{w}\circ p$, the canonical extensions of $f$ and $\bs{w}$, 
respectively, where $p$ is the projection given by (\ref{eq:projection}). 
Then
\begin{align*}
\nabla(f\circ p)&=[I-b\,D^2b]\gradg f\circ p, \qquad \qquad
D(\bs{w}\circ p) =\Dg\bs{w}\circ p\,[I-b\,D^2b],\\
\Div(\bs{w}\circ p)&=[I-b\,D^2b]:\Dg\bs{w}\circ p
= \Divg \bs{w}\circ p -b D^2b: \Dg\bs{w}\circ p.
\end{align*}
In particular,
\begin{equation}\label{eq:tan-der-def3}
\gradg f =  \nabla(f\circ p)|_\Gamma, \qquad
\Dg\bs{w} =  D(\bs{w}\circ p)|_\Gamma,\qquad
\Divg\bs{w} =  \Div(\bs{w}\circ p)|_\Gamma.
\end{equation} 
\end{lmm} 
The tangential divergence of a tensor valued function $S$ is defined as
$\Divg S \cdot \bs{e} = \Divg (S^T\bs{e})$, for any vector $\bs{e}$.
The expressions (\ref{eq:tan-der-def3}) of tangential derivatives given by 
canonical extensions allow us to prove directly the following product rule 
formulas, already known for classical derivatives~\cite[p.~30]{Gur1981}.
\begin{lmm}[Product Rule for tangential derivatives]\label{le:DgProperties}
Let $\alpha$, $\bs{u}$, $\bs{v}$ and $S$ be smooth fields in $\Gamma$, 
with $\alpha$ scalar valued, $\bs{u}$ and 
$\bs{v}$ vector valued, and $S$ tensor valued. Then 
\begin{multicols}{2}
\begin{enumerate}[(i)]
\item $\Dg(\varphi \bs{u})= \bs{u}\otimes\nabla_\Gamma\varphi + 
  \varphi\Dg\bs{u}$,
\item $\Divg(\varphi\bs{u})=\varphi \Divg \bs{u} + \bs{u}\cdot 
  \nabla_\Gamma \varphi$
\item $\gradg (\bs{u}\cdot \bs{v}) =\Dg \bs{u}^T\bs{v} + \Dg\bs{v}^T\bs{u}$
\item $\Divg(\bs{u}\otimes\bs{v})=\bs{u}\Divg\bs{v}+D_\Gamma\bs{u}\,\bs{v}$,
\item $\Divg(S^T\bs{u})=S: \Dg \bs{u} + \bs{u}\cdot \Divg S$.
\item $\Divg(\alpha S)=S\gradg \alpha + \alpha\, \Divg S$.
\end{enumerate}
\end{multicols}
\end{lmm}

It will be very useful for us to write the geometric invariants (see 
Section \ref{sec:invariants}) of $\Gamma$ in terms of tangential derivatives of 
the normal vector field $\bs{n}$. 
The tensor $-\Dg\bs{n}(x)$ (see \cite[Ch.~1.3]{Giga2006}) defined from the 
tangent plane $T_x(\Gamma)$ to itself, is called the \emph{Weingarten map} and 
is associated to the second fundamental form of $\Gamma$.
Since $\bs{n}\circ p=\nabla b $, (\ref{eq:tan-der-def3}) implies $\Dg\bs{n} = 
D(\n \circ p)|_\Gamma = D^2b|_\Gamma$, so that 
\begin{equation}\label{mcurvature2} \kappa = \Delta b|_\Gamma = \tr(D^2 
b|_\Gamma) = \tr(\Dg \n) = \Divg \bs{n}, \end{equation} and $\sum \kappa_i^2 = 
|D^2b|_\Gamma|^2 = |\Dg\bs{n}|^2$ whence, for $N=3$, 
\begin{equation}\label{gcurvature2} 
\kappa_g= \frac{1}{2}\left(\kappa^2 - 
|\Dg\bs{n}|^2\right). 
\end{equation}
In, particular, as we will see in Section~\ref{sec:invariants}, any geometric 
invariant can be written in terms of $I_p\definedas\tr(\Dg\bs{n}^p)$.
    
The Divergence Theorem for surfaces whose proof can be found in
(Prop. 15 of \cite{Wal2015}) is the following
\begin{lmm}[Tangential Divergence Theorem]\label{le:div-teo}
If $\Gamma =\partial \Omega$ is $\mathcal{C}^2$ and 
$\bs{w}\in\mathcal{C}^1(\Gamma, \Rn[N])$, then
\begin{equation}\label{eq:tan-div-teo-boundary}
\int_\Gamma \Divg\bs{w} = \int_\Gamma \kappa\, \bs{w}\cdot\bs{n},
\end{equation}
where $\kappa$ is the mean curvature of $\Gamma$ and $\bs{n}$ its normal 
field.
If $\Gamma \subsetneq \partial \Omega$, then
\begin{equation}\label{eq:tan-div-teo-surface}
\int_\Gamma \Divg\bs{w} = 
\int_\Gamma \kappa\, \bs{w}\cdot\bs{n}+ \int_{\partial\Gamma}\bs{w}\cdot 
\bs{n}_s,
\end{equation} 
where $\bs{n}_s$ is the outward normal to $\partial\Gamma$ which is also 
normal to $\bs{n}$.
\end{lmm}   
    
The following Lemma is new and extends formula (\ref{eq:div-formula}) for 
tangential derivatives.  
\begin{lmm}\label{le:div-formula-tan}
If $\Gamma$ is $\mathcal{C}^3$ and $\bs{w}\in \mathcal{C}^3(\Gamma,\Rn[N])$, 
we have 
$\gradg \Divg \bs{w} = P \Divg \Dg\bs{w}^T-\Dg\bs{n}\Dg\bs{w}^T\bs{n}$,
where $P= I-\bs{n}\tensor\bs{n}$ is the orthogonal projection operator 
given by (\ref{eq:vector-projection}).
\end{lmm}
       
       \begin{proof}
       We use formula (\ref{eq:tan-der-def3}) to write tangential derivatives 
using the projection function $p$:
       \begin{equation*}
       \gradg \Divg \bs{w} = \nabla (\Divg \bs{w} \circ p) |_\Gamma 
= \nabla (\Div (\bs{w} \circ p)\circ p ) 
|_\Gamma.
\end{equation*}
       Then we use successively the chain rule, the derivative of $p$ given by 
(\ref{eq:projection-derivative}) and the property of classical
       derivatives
       (\ref{eq:div-formula}):
       \begin{equation}\label{eq:bla}
       \gradg \Divg \bs{w}= Dp^T |_\Gamma \nabla \Div (\bs{w} \circ p) 
|_\Gamma 
       = P\ \nabla \Div (\bs{w}\circ p) |_\Gamma 
       = P\ \Div (D(\bs{w}\circ p)^T) |_\Gamma.
       \end{equation}
       
       Note that Lemma \ref{le:canonical-ext} implies $D(\bs{w}\circ p)^T 
=\Dg\bs{w}^T\circ p - b\, D^2b\, (\Dg\bs{w}^T\circ p)$, 
       and the product rule $\Div(\alpha S)= \alpha\, \Div S + S\nabla \alpha $ 
implies
       \begin{equation*}
       \Div(D(\bs{w}\circ p)^T)= \Div(\Dg\bs{w}^T\circ p)-b\, \Div(D^2b\, 
\Dg\bs{w}^T\circ p) - 
        D^2b\, \Dg\bs{w}^T\circ p\,\nabla b.
       \end{equation*}
       Then, after restricting to $\Gamma$ we have
       $
       \Div(D(\bs{w}\circ p)^T)|_\Gamma = 
\Divg(\Dg\bs{w}^T)-\Dg\bs{n}\Dg\bs{w}^T \bs{n},
       $
       which implies, from (\ref{eq:bla}), the desired result.
       \end{proof}

Applying Lemma~\eqref{le:div-formula-tan} to $\bs{w}=\n$ we obtain for 
$\kappa = \Divg \n$
\begin{equation}\label{eq:grad_curvature}
\nabla_\Gamma \kappa = \nabla_\Gamma \Divg \n 
= P \Divg \nabla_\Gamma \n - \Dg \n \Dg \n^T \n 
= P \lb \n
\end{equation}       
because $\Dg \n^T \n = 0$ and $\lb = \Divg \nabla_\Gamma$.
       
%

\section{Shape Functionals and Derivatives}\label{S:functionals}
A \emph{shape functional} is a function $J:\mathcal{A}\to\R$ defined on a 
set $\mathcal{A}=\mathcal{A}(\mathbf{D})$ of admissible 
subsets of a hold-all domain $\mathbf{D}\subset \Rn[N]$.

Let the elements of $\mathcal{A}$ be smooth domains
and for each $\Omega \in \mathcal{A}$, let $y(\Omega)$ be a 
function in $W(\Omega)$ some Sobolev space over $\Omega$. 
Then the shape functional given by
$J(\Omega)= \int_{\Omega} y(\Omega)(x) dx = \int_{\Omega} y(\Omega)$
is called a \emph{domain functional}.
For example the \emph{volume functional} is obtained with $y(\Omega) \equiv 1$, 
but the \emph{domain function} $y(\Omega)$ could be something more involved such as the solution of a PDE 
in $\Omega$.

Our main interest in this work are the \emph{boundary functionals} 
given by
$J(\Gamma)= \int_{\Gamma} z(\Gamma)(x) \, d\Gamma = \int_{\Gamma} z(\Gamma)$,   
where $z$ is a function that for each surface $\Gamma$ in a family of 
admissible surfaces $\mathcal{A}$ assigns a function $z(\Gamma) \in W(\Gamma)$, 
with 
$W(\Gamma)$ some Sobolev space on $\Gamma$. The area functional corresponds to 
$z(\Gamma) \equiv 1$, but more interesting functionals are obtained when the \emph{boundary function}
$z(\Gamma)$ depends on the mean curvature $\kappa$ of $\Gamma$ or on the 
geometric invariants $I_p(\Gamma)=\tr(\Dg\bs{n}^p)$, with $p$ a positive integer, 
or any real function which involves the normal field $\bs{n}$ or higher order 
tangential derivatives on $\Gamma$. 
    
\subsection{The velocity Method}  
On a hold-all domain $\D$ (not necessarily bounded), we call an 
\emph{autonomous velocity} 
to a vector field $\bs{v}\in V^k(\D) \definedas\mathcal{C}_0^k(\D,\Rn[N])$, 
the set of all $\mathcal{C}^k$ functions $f$ such that 
$D^\alpha f$ has compact support contained in $\D$, for $0\leq|\alpha|\leq k$;
hereafter we assume that $k$ is a fixed positive integer.
A \emph{(nonautonomous) velocity field} 
$\bs{V}\in\mathcal{C}([0,\epsilon],V^k(\D))$,  (Theorem 2.16 of \cite{SZ1992}) 
induces a \emph{trajectory} $x=x_{\bs{V}} \in 
\mathcal{C}^1([0,\epsilon], V^k(\D))$, through the system of ODE      
\begin{equation}\label{def:trajectories} 
\dot{x}(t)=\bs{V}(t)\circ x(t),\ t\in[0,\epsilon],\qquad
x(0)=id,  \\
\end{equation}
where we use a point to denote derivative respect to the time variable $t$.
\begin{rmrk}[Initial velocity]
We call $\bs{v}$ to the velocity field at $t=0$, namely 
$\bs{v} = \bs{V}(0)$. In the autonomous case, $\bs{V}(t)= \bs{v}$ for any $t$, 
with $\bs{v}\in V^k(\D)$, and the trajectory $x(t)$ is given by
\begin{equation}\label{eq:autonomous trajectories}
\dot{x}(t)=\bs{v}\circ x(t), \  t\in[0,\epsilon],\qquad
x(0)=id. 
\end{equation}
\end{rmrk}
   
\subsection{Shape Differentiation}

Given a velocity field $\bs{V}$ and a subset $S\subset \D$, the perturbed 
set at time $t$ is given by $S_t=x(t)(S)$, 
where $x(t)$ is the trajectory given by (\ref{def:trajectories}).
For a shape functional $J:\Ad \to \R$, where $\Ad$ is a family 
of admissible sets $S$ (domains or boundaries), 
and a velocity field $\bs{V}\in \mathcal{C}([0,T], V^k(\D))$, the 
\emph{Eulerian semiderivative} of $J$ 
at $S$ in the direction $\bs{V}$ is given by
\begin{equation}\label{def:eulerian-der}
d J(S; \bs{V}) = \lim_{t\searrow 0} \frac{J(S_t)-J(S)}{t},
\end{equation}
whenever the limit exist.

\begin{dfntn}[Shape differentiable]
We say that $J$ is shape differentiable 
at $S$ when the Eulerian semiderivative (\ref{def:eulerian-der}) exists for any 
$\bs{V}$ in the vector space of velocities $ \mathcal{C}([0,T], V^k(\D))$, 
and the functional
$\bs{V} \to d J(S; \bs{V})$ is linear and continuous. 
\end{dfntn}
\begin{rmrk}[Hadamard differentiable]\label{nt:had-dif}
By Theorem 3.1 of \cite[Ch.~9]{DZ2011}, if a functional $J$ is shape 
differentiable 
then $dJ(S,\bs{V})$ depends only on $\bs{v}\definedas\bs{V}(0)$ (that is, 
$J$ is Hadamard differentiable at 
$S$ in the direction $\bs{v}$).
\end{rmrk}
\begin{dfntn}[Shape derivative]\label{d:shape_der}
If $J$ is shape differentiable we call $dJ$  
its shape derivative.
\end{dfntn}
\begin{rmrk}[Taylor formula]\label{r:taylor}
Given $\bs{V} \in \mathcal{C}([0,T], V^k(\D))$ we can define
$S+\bs{V}$ to be $S_t$ for $t=1$ provided it is admissible.
Then if $J$ is shape differentiable (see \cite[Ch.~9]{DZ2011}) it follows that
$J(S + \bs{V}) = J(S)+d J(S; \bs{V}) + o(|\bs{V}|)$.
\end{rmrk}

\subsection{The Structure Theorem}\label{sec:structure_theorem}
One of the main results about shape derivatives is the 
(Hadamard-Zolesio) Structure Theorem (Theorem 3.6 of \cite[Ch.~9]{DZ2011}).
It establishes that, if a shape functional $J$ is shape differentiable at 
the domain $\Omega$ with boundary $\Gamma$, then the only relevant part of the 
velocity field $\bs{V}$ in $dJ(\Omega, \bs{V})$ 
is $v_n \definedas \bs{V}(0)\cdot \bs{n}|_\Gamma$. In other words, if 
$\bs{V}(0)\cdot \bs{n}=0 $ in $\Gamma$, then  $dJ(\Omega, \bs{V})=0$.
More precisely, 
\begin{thrm}[Structure Theorem]\label{t:structure}
Let $\Omega\in \Ad$ be a domain with $\mathcal{C}^{k+1}$-boundary 
$\Gamma$, 
$k\geq 0$ integer, and let $J:\Ad\rightarrow \R$ be a shape functional
which is shape differentiable at $\Omega$ with respect to $V^k(\D)$. Then 
there exists a functional $g(\Gamma)\in \left(\mathcal{C}^k(\Gamma)\right)'$ 
(called the shape gradient) such that
$dJ(\Omega, \bs{V}) =\langle g(\Gamma),v_n\rangle_{\mathcal{C}^k(\Gamma)}$,
where $v_n = \bs{V(0)}\cdot\bs{n}$.
Moreover, if the \textit{gradient} $g(\Gamma)\in L^1(\Gamma)$, then
$dJ(\Omega, \bs{V})=\int_\Gamma g(\Gamma)\,v_n$.
\end{thrm}

\section{Shape Derivatives of Domain Functions}\label{S:shape-domain}
In this section and the following, we find specialized formulas
for the shape derivatives of domain and boundary functionals.
These, in turn, will induce definitions for the shape derivatives of domain and 
boundary functions.
    
\subsection{Shape differentiation of a domain functional}
Consider a  velocity field $\bs{V}\in \mathcal{C}([0,\epsilon], V^k(\D))$,  
$k \geq 1$, with trajectories $x\in \mathcal{C}^1([0,\epsilon], V^k(\D))$ 
satisfying (\ref{def:trajectories}), and note that the Eulerian semiderivative 
(\ref{def:eulerian-der}) can be written as $d J(\Omega; 
\bs{V})=\frac{d}{dt^{+}}J(\Omega_t)|_{t=0}$, where $\Omega_t=x(t)(\Omega)$ and 
$\Omega_0=\Omega$.  Then, from a well known change of variables formula 
\cite[Ch.~9, Sec.~4.1]{DZ2011}, we have
\begin{equation}\label{eq:JOmegat-change}
J(\Omega_t)=\int_{\Omega_t}y(\Omega_t)\  d\Omega_t = \int_{\Omega} 
[y(\Omega_t)\circ x(t)]\,\gamma(t)\,d\Omega
\end{equation}
where $\gamma(t)\definedas \det Dx(t)$, with $Dx(t)$ denoting derivative 
with respect to the spatial variable $X$. Note that if $y(\Omega_t) \in
W^{r,p}(\Omega_t)$ for each $t$ then  $y(\Omega_t)\circ x(t)\in 
W^{r,p}(\Omega)$ (if $0\le r \le k$).
The following Lemma (Theorem 4.1 in Ch.~9, p.~482 of 
\cite{DZ2011}) provides some insight on the nature of $\dot\gamma$.
\begin{lmm}[Time derivative of $\gamma$]\label{le:der-gamma}
If $\bs{V} \in \mathcal{C}([0,\epsilon], V^1(\D))$ then 
$\gamma \in \mathcal{C}^1([0,\epsilon], V^1(\D)),$
and its (time) derivative is given by
$\dot{\gamma}(t) =\gamma(t) \left[\Div\bs{V}(t) \circ x(t)\right]$.
In particular, $\dot{\gamma}(0) = \Div\bs{v}$, where $\bs{v}=\bs{V}(0)$.
\end{lmm}

In order to motivate the definition of material and shape derivative, consider the following situation.
Let $y$ be a domain function which assigns a function $y(\Omega) \in W(\Omega)$ to each domain $\Omega$ in a class $\Ad$ of admissible smooth domains.
Suppose that the function $f:[0,\epsilon]\to L^1(\Omega)$ given by 
$f(t) = y(\Omega_t)\circ x(t)$
is differentiable at $t=0$ in $L^1(\Omega)$, that is, there exists 
$\dot{f}(0)\in L^1(\Omega)$ such that
$\lim_{t\searrow 0}\|\frac{f(t)-f(0)}{t}-\dot{f}(0)\|_{L^1(\Omega)}=0$.
Then we can differentiate inside the integral (\ref{eq:JOmegat-change}) to 
obtain
\begin{equation}\label{eq:derivada-adentro}
\frac{d}{dt}J(\Omega_t)|_{t=0}= \int_{\Omega} \dot{f}(0) \gamma(0) + 
f(0)\dot{\gamma}(0).
\end{equation}

Finally, using Lemma \ref{le:der-gamma}, and that $\gamma(0)=1$ and 
$f(0)=y(\Omega)$, we obtain
\begin{equation}
dJ(\Omega, \bs{V})= \int_\Omega \dot{f}(0) + y(\Omega) \Div\bs{v},
\end{equation}
where $\bs{v}=\bs{V}(0)$.
%
    
\subsection{Material and shape derivatives}

\begin{dfntn}[Material Derivative (Def.~2.71, 
p.~98 of~\cite{SZ1992})]\label{def:material-der}
Consider a velocity vector field $\bs{V}\in 
\mathcal{C}\left([0,\epsilon],V^k(\D)\right)$, with $k\geq 1$,
an admissible set $S\subset \D$ (domain or boundary) of class 
$\mathcal{C}^k$, and a function $y(S) \in W^{r, p}(S)$, 
with $r\in(0,k]\cap \mathbb{Z}$. Suppose there exists $y(S_t) \in W^{r, 
p}(S_t)$ for all $0<t<\epsilon$,
where $S_t=x(t)(S)$ is the perturbation set of $S$ by the trajectories 
$x(t)$ given by $\bs{V}$. 
The \emph{material derivative} of $y(S)$ at $S$ in the direction 
$\bs{V}$
is the function  $\dot{y}(S, \bs{V})\in W^{r-1, p}(S)$, given by 
\begin{equation}\label{def:eq:material-der}
\dot{y}(S, \bs{V}) = \frac{d}{dt^+} [y(S_t)\circ x(t)]_{t=0} 
  =  \lim_{t\searrow 0} \frac{y(S_t)\circ x(t) - y(S)}{t},
\end{equation}
whenever the limit exists in the sense of $W^{r-1, p}(S)$. 
In this case we say that the material derivative of $y(S)$ exists at $S$ 
in $W^{r-1, p}(S)$ in the direction $\bs{V}$.
We can replace the space $W^{r, p}(S)$ by $\mathcal{C}^r(\Omega)$, 
$1\leq r \leq k$, obtaining 
$\dot{y}(S, \bs{V})\in \mathcal{C}^{r-1}(\Omega)$.
\end{dfntn}

With this definition, the existence of material derivative of $y(\Omega)\in 
W^{1,1}(\Omega)$ implies the differentiability of $f(t)$ at $t=0$ in 
$L^1(\Omega)$, which was the assumption needed for equation 
(\ref{eq:derivada-adentro}) to hold. Then, for $J(\Omega)=\int_{\Omega} 
y(\Omega)$,  we have
\begin{equation}\label{eq:first_integral_form}
dJ(\Omega, \bs{V}) = \int_\Omega \dot{y}(\Omega, \bs{V}) +  y(\Omega) 
\Div\bs{v}.
\end{equation}
   
\begin{rmrk}[Autonomous dependence]
If $J(\Omega)=\int_{\Omega} y(\Omega)$ is shape differentiable at 
$\Omega$ and 
$\dot{y}(\Omega, \bs{V})$ exists for any velocity $\bs{V}$, we obtain 
from Remark \ref{nt:had-dif} and equation (\ref{eq:first_integral_form})
that $\dot{y}(\Omega, \bs{V})=\dot{y}(\Omega, \bs{v})$, where 
$\bs{v}=\bs{V}(0)$.
\end{rmrk}
As a particular case, suppose that $y(\Omega)$ is independent of the 
geometry, namely: $y(\Omega)=\phi|_\Omega$, with 
$\phi\in W^{1,1}(\D)$. 
Then, by the chain rule,    
$\dot{y}(\Omega, \bs{V}) = \nabla \phi \cdot \bs{v}$, and we have
$
dJ(\Omega, \bs{V}) = \int_\Omega \nabla \phi \cdot \bs{v} +  \phi\, 
\Div\bs{v}
=\int_\Omega \Div\left(\phi\,\bs{v}\right)$.
If the boundary $\Gamma=\partial \Omega \in \mathcal{C}^1$ then, the 
Divergence Theorem yields
$d J(\Omega; \bs{V}) = \int_\Gamma \phi \, v_n\ d\Gamma$.    
This dependence of $dJ(\Omega,\bs{V})$ on $\bs{V}$ only 
trough the normal component of $\bs{v}$ in $\Gamma$ was expected by the 
Structure Theorem.
    
In the general case, when $y(\Omega)\in W^{1,1}(\Omega)$ depends on the 
geometry of $\Omega$, from
(\ref{eq:first_integral_form}) we can write
\begin{equation*}
dJ(\Omega, \bs{V}) = \int_\Omega \left[\dot{y}(\Omega, \bs{V}) - \nabla 
y(\Omega) \cdot \bs{v}\right] + \Div\left(y(\Omega)\bs{v}\right).
\end{equation*}
Which leads to the following definition of \emph{shape derivative of a 
domain function}.
\begin{dfntn}[Shape derivative of a domain 
function]\label{def:function-domain-shape-der}  
Given a velocity field $\bs{V}\in \mathcal{C}([0,\epsilon], V^k(\D))$, 
if there exists the material derivative of 
$y(\Omega)\in W^{r,p}(\Omega)$, with $1\leq r\leq k$, then the 
\emph{domain shape derivative} $y'(\Omega, \bs{V})\in W^{r-1,p}(\Omega)$  
is given by
\begin{equation}\label{def:eq:function-domain-shape-der}
y'(\Omega, \bs{V}) = \dot{y}(\Omega, \bs{V}) - \nabla y(\Omega)\cdot 
\bs{v},
\end{equation}
where $\bs{v}=\bs{V}(0)$.
\end{dfntn}
If $y'(\Omega,\bs{V}) \in W^{r-1,p}(\Omega)$ exists, then we have
\begin{equation}\label{eq:integral-domain-shape-der2}
d J(\Omega; \bs{V}) = \int_\Omega y'(\Omega, \bs{V}) + 
\Div\left(y(\Omega)\bs{V}(0)\right) 
=\int_\Omega y'(\Omega, \bs{V}) + \int_\Gamma y(\Omega)\,v_n,
\end{equation}
whenever $\Gamma$ is $ \mathcal{C}^1$. 
    
\section{Shape derivative of boundary functions}\label{S:shape-boundary}
Consider now a boundary functional of the form
$J(\Gamma)=\int_{\Gamma} z(\Gamma) d\Gamma$,
where $\Gamma =\partial \Omega$  is a boundary which is also a 
$\mathcal{C}^k$-manifold, and for each admissible $\Gamma$, $z(\Gamma)\in 
W^{r,p}(\Gamma)$, with $1\leq r \leq k$.
Below we derive a formula for the shape derivative $dJ(\Gamma, \bs{V})$ 
analogous to~(\ref{eq:integral-domain-shape-der2})
which uses the concepts from tangential calculus given in Section 
\ref{sec:tangential}.
\subsection{Shape differentiation of a boundary functional}
    
For $\Gamma_t\definedas x(t)(\Gamma)$, we have that
\begin{equation}\label{eq:J-en-Gamma-t}
J(\Gamma_t)=\int_{\Gamma_t} z(\Gamma_t) d\Gamma_t 
= \int_{\Gamma} [z(\Gamma_t)\circ x(t)]\, \omega(t)\, d\Gamma,
\end{equation}
where $\omega(t)= \|M(Dx(t))\bs{n}\|=\gamma(t)\|Dx(t)^{-T}\bs{n}\|$,
with $\gamma(t)=\det{Dx(t)}$ and $M(Dx(t))= \gamma(t)Dx(t)^{-T}$ is the 
cofactor matrix of the Jacobian $Dx(t)$. This well known change of variable 
formula can be found in Proposition 2.47 of~\cite[p.~78]{SZ1992}.
We can also find there a formula for the derivative of $\omega(t)$ at 
$t=0$, which we now cite using the notation of tangential divergence.
\begin{lmm}[Derivative of $\omega$ (Lemma 2.49, p. 80, of \cite{SZ1992})]
The mapping $t\rightarrow \omega(t)$ is differentiable from 
$[0,\epsilon]$ into $\mathcal{C}^{k-1}(\Gamma)$, 
and the derivative at $t=0$ is given by
\begin{equation*}
\dot{\omega}(0)=\Divg \bs{v},  
\end{equation*}
where  $\Divg \bs{v}=\Div \bs{v} - D\bs{v}\bs{n}\cdot\bs{n}$ is the 
tangential divergence of
$\bs{v}=\bs{V}(0)$.
\end{lmm}

    Assuming that the material derivative $\dot{z}(\Gamma, \bs{V})$, given by 
Definition \ref{def:material-der},  
    exists in $L^1(\Gamma)$, we can differentiate inside the integral in 
(\ref{eq:J-en-Gamma-t}) and then obtain
    \begin{equation*}
    dJ(\Gamma, \bs{V})=\int_\Gamma \dot{z}(\Gamma, \bs{V}) + 
z(\Gamma)\,\Divg\bs{v}.
    \end{equation*}
    Using the product rule formula $\Divg (z\bs{v})=z\,\Divg\bs{v}+\bs{v}\cdot 
\gradg z$ (Lemma \ref{le:DgProperties}) we can write
    \begin{equation*}
    dJ(\Gamma, \bs{V})=\int_\Gamma \left[\dot{z}(\Gamma, \bs{V})-\gradg 
z(\Gamma)\cdot \bs{v}\right] + \Divg(z(\Gamma)\bs{v}).
    \end{equation*}

We are now in position to introduce the concept of \emph{shape derivative 
of a boundary function}.

\begin{dfntn}[Shape Derivative of a boundary function] 
\label{def:shape-der-boundary-scalar}  
Let $z$ be a boundary function which satisfies  $z(\Gamma) \in 
W^{r,p}(\Gamma)$ for all $\Gamma$ in an admissible set $\mathcal{A}$ of 
boundaries of class $\mathcal{C}^k$.
If the material derivative $\dot z(\Gamma,\bs{V})$ exists in 
$W^{r-1,p}(\Gamma)$ (Definition~\ref{def:material-der}) for a velocity  
$\bs{V}\in  \mathcal{C}([0,\epsilon], V^k(\D)) $,
then the shape derivative $z'(\Gamma, \bs{V})\in W^{r-1,p}(\Gamma)$  is given by
$z'(\Gamma, \bs{V}) = \dot{z}(\Gamma, \bs{V}) - \gradg z(\Gamma)\cdot 
\bs{v}$, where $\bs{v}=\bs{V}(0)$.
\end{dfntn}
    
    Note that $\dot{z}(\Gamma, \bs{V})$ and $z'(\Gamma, \bs{V})$ are intrinsic 
to 
$\Gamma$, because they do not depend on any extension of $z(\Gamma)$ to an open 
set containing $\Gamma$. With this notation, for 
$J(\Gamma)=\int_{\Gamma}z(\Gamma) \, d\Gamma$ we have
    $ d J(\Gamma; \bs{V}) =\int_\Gamma z'(\Gamma, \bs{V}) + 
\Divg\left(z(\Gamma)\bs{v}\right)$.
  
    If $\Gamma\in\mathcal{C}^2$, then the tangential divergence formula 
(\ref{eq:tan-div-teo-boundary})  of Lemma \ref{le:div-teo} gives us the 
expression
    \begin{equation}\label{eq:integral-boundary-shape-der2}
    d J(\Gamma; \bs{V}) =\int_\Gamma z'(\Gamma, \bs{V}) + \kappa\, 
z(\Gamma)\,v_n ,
    \end{equation}
    where  $\kappa=\kappa(\Gamma)$ is the mean curvature function on $\Gamma$, 
and $v_n = \bs{v}\cdot \bs{n}$.

\begin{rmrk}\label{rem:surfaces-velocities}
For a surface $\Gamma \subset \partial \Omega$, we can consider the 
space of velocity fields 
$V_\Gamma^k(\D) = V^k(\D)\bigcap\{\bs{v}:\bs{v}|_{\partial \Gamma}=0\}$, 
in order to obtain
(\ref{eq:integral-boundary-shape-der2}) by applying formula 
(\ref{eq:tan-div-teo-surface}) of 
Lemma \ref{le:div-teo}. 
\end{rmrk}

\subsection{Relation between domain and boundary function}

Suppose that $z(\Gamma)$ is the restriction of a domain function $y(\Omega)$ to 
its boundary, that is: $z(\Gamma) = y(\Omega)|_\Gamma$, 
where $\partial \Omega = \Gamma $ and suppose that $y(\Omega)$ is shape 
differentiable at $\Omega$ in the direction $\bs{V}$. 
The \emph{material} derivatives of both $z(\Gamma)$ and $y(\Omega)$
are the same, as is established in the following Lemma.

\begin{lmm}[Proposition 2.75 of 
\cite{SZ1992}]\label{le:material-derivative-equal}
 Assume that the material derivative of a domain function $y(\Omega)\in 
W^{r,p}(\Omega)$ exists at the domain $\Omega$ of class $\mathcal{C}^k$ in the 
direction of a velocity 
 field $\bs{V}\in  \mathcal{C}([0,\epsilon], V^k(\D)) $, and that 
$\dot{y}(\Omega,\bs{V})\in W^{r-1,p}(\Omega)$. Then for $r>\frac{1}{p}+1$, 
there 
exists the material derivative of 
 $z(\Gamma)=y(\Omega)|_\Gamma$ at $\Gamma$ in the direction $\bs{V}$, and
 \begin{equation*}
 \dot{z}(\Gamma,\bs{V})= \dot{y}(\Omega,\bs{V})|_\Gamma\in 
W^{r-1-\frac{1}{p},p}(\Omega)
 \end{equation*}
\end{lmm}

However, the \emph{shape} derivatives of domain and boundary functions are not 
generally equal for the same function, as is shown in the next lemma. 

\begin{lmm}[Domain and boundary 
function]\label{le:domain-boundary-shapeder-scalar}
Consider a domain $\Omega$ with boundary $\Gamma$, and functions $z(\Gamma)$ 
and 
$y(\Omega)$ such that  $z(\Gamma) = y(\Omega)|_\Gamma$. Then, $z(\Gamma)$ is 
shape differentiable in $\Gamma$ at direction $\bs{V}\in  
\mathcal{C}([0,\epsilon], V^k(\D)) $ if 
$y(\Omega)$ is shape 
differentiable in $\Omega$ at direction $\bs{V}$, and
\begin{equation*}
z'(\Gamma, \bs{V})= y'(\Omega, \bs{V})|_\Gamma + \frac{\partial 
y(\Omega)}{\partial \bs{n}} \,v_n,\qquad\text{where $v_n = \bs{v}\cdot \bs{n}$.}
\end{equation*} 
\end{lmm}   
\begin{proof}
 Since $y(\Omega)$ is an extension to $\Omega$ of $z(\Gamma)$, we have $\gradg 
z= \nabla y|_\Gamma-\frac{\partial y}{\partial n}\bs{n}$. 
 Then, using Lemma \ref{le:material-derivative-equal} and comparing Definitions 
\ref{def:function-domain-shape-der}  and \ref{def:shape-der-boundary-scalar}, 
we 
obtain the desired result.
\end{proof}

    Going back to formula (\ref{eq:integral-boundary-shape-der2}), if 
$z(\Gamma)$ admits an extension $y(\Omega)$,
    Lemma \ref{le:domain-boundary-shapeder-scalar} gives us
    \begin{equation*}
    d J(\Gamma; \bs{V}) =\int_\Gamma y'(\Omega, \bs{V})|_\Gamma + \left( 
\frac{\partial y(\Omega)}{\partial \bs{n}} +\kappa z(\Gamma)\right)v_n.
    \end{equation*}

    \section{Properties of shape derivatives of domain functions} 
\label{S:domain-properties}

    We first extend the Definition \ref{def:function-domain-shape-der} of 
domain 
shape derivative to vector and tensor valued functions.
    
    \begin{dfntn}[Vector and tensor valued domain 
functions]\label{def:domain-shape-der-vector}    
       Consider a vector field $\bs{w}(\Omega)\in W^{r,p}(\Omega, \Rn[N])$ 
which 
exists for all admissible domains $\Omega\in \Ad(\D)$. 
   For a given velocity field $\bs{V}$, we say that $\bs{w}(\Omega)$ 
   is shape differentiable at $\Omega$ in the direction of $\bs{V}$ if there 
exists the material derivative 
   $\dot{\bs{w}}(\Omega, \bs{V}) = \frac{d}{dt} [\bs{w}(\Omega_t)\circ 
x(t)]_{t=0}$ in $W^{r-1,p}(\Omega, \Rn[N])$.
   In that case the (domain) shape derivative  belongs to
   $W^{r-1,p}(\Omega, \Rn[N])$ and is given by
   \begin{equation*}
    \bs{w}'(\Omega, \bs{V})= \dot{\bs{w}}(\Omega, \bs{V}) - D\bs{w}(\Omega)\bs{v},
   \end{equation*}
where $\bs{v}=\bs{V(0)}$.   
If $N=1$, we consider $D\bs{w}=(\nabla \bs{w})^T$.

For a tensor valued function $A(\Omega):\Omega \rightarrow \Lin(\mathbb{V})$, 
we 
say that $A$ is shape differentiable at $\Omega$ in the direction of $\bs{V}$ 
if 
so is the vector valued function $A\bs{e}$, for any vector $\bs{e}\in 
\mathbb{V}$. The shape derivative $A'(\Omega, \bs{V})$ is the tensor valued 
function which satisfies
\begin{equation}\label{eq:def-tensorshape}
A'(\Omega, \bs{V})\bs{e} =(A\bs{e})'(\Omega, \bs{V})\ \ \ \text{for any }
\bs{e}\in \mathbb{V}.
\end{equation}
Throughout this paper we consider $\mathbb{V}=\Rn[N]$. 
\end{dfntn}

     \begin{lmm}\label{le:properties_domain}
     For a given admissible domain $\Omega \subset \D$ with boundary 
$\mathcal{C}^k$, $k\geq 2$, a velocity field $\bs{V}\in 
     \mathcal{C}([0,\epsilon], V^k(\D))$ and  a shape differentiable function 
$y(\Omega)\in W^{r,p}(\Omega)$, $1\leq r \leq k$, $1\leq p<\infty$,
     we have the following properties:
   \begin{enumerate}
   
   \item If the mapping $\bs{V}\rightarrow \dot{y}(\Omega, \bs{V})$ is 
continuous from $ \mathcal{C}([0,\epsilon], V^k(\D))$ into $W^{r-1,p}(\Omega)$, 
then $y'(\Omega, \bs{V})=y'(\Omega, \bs{v})$, where $\bs{v}=\bs{V}(0)$
  (Proposition 2.86 of \cite{SZ1992}).
  
   \item Suppose that $\bs{V}\rightarrow \dot{y}(\Omega, \bs{V})$ is continuous 
from $ \mathcal{C}([0,\epsilon], V^k(\D))$ into $W^{r-1,p}(\Omega)$. If the 
velocity fields $\bs{V}_1$ and $\bs{V}_2$ are such that $\bs{V}_1(0)\cdot 
\bs{n} 
= \bs{V}_2(0)\cdot \bs{n}$ on $\Gamma=\partial \Omega$, then $y'(\Omega, 
\bs{V_1})=y'(\Omega, \bs{V_2})$ (Proposition 2.87 of \cite{SZ1992}).
      \item  If $y(\Omega)=\phi|_\Omega$, for $\phi\in W^{r,p}(\D)$, then $y$ 
is 
shape differentiable in $\Omega$ for any direction $\bs{V}$, and $y'(\Omega, 
\bs{V})=0$ (Proposition 2.72 of \cite{SZ1992}).
      
   \end{enumerate}
\end{lmm}

   \begin{lmm}[Lemma 4 from \cite{HR2003}]\label{le:y_zero_Gamma}
   Suppose that $y(\Omega) \in H^{\frac{3}{2}+\epsilon}(\Omega)$ satisfies 
$y(\Omega)|_\Gamma = 0$ for all
domains $\Omega\in \mathcal{A}$ and that the shape derivative $y'(\Omega; 
\bs{V})$ exists in $H^{\frac{1}{2} +\epsilon}(\Omega)$ for some $\epsilon > 0$.
Then, we have
    \begin{equation}\label{eq:y_zero_Gamma}
     y'(\Omega, \bs{V})|_\Gamma =  -\frac{\partial y}{\partial\bs{n}}\, v_n 
    \end{equation}
    where $v_n=\bs{v}\cdot \bs{n}$ and $\bs{v}=\bs{V}(0)$.
    
   \end{lmm}
\begin{proof}
This Lemma is proved in  \cite{HR2003}. However, it can be also 
demonstrated if we consider the boundary function $z(\Gamma)\definedas 
y(\Omega)|_\Gamma$. In fact, by hypothesis, $z(\Gamma_t)\equiv 0$ for all small 
$t\geq0$. This gives us $\dot{z}(\Gamma, \bs{V})=0$ and $\gradg z(\Gamma)=0$, 
so 
that $z'(\Gamma, \bs{V})=0$. The claim thus follows from Lemma~\ref{le:domain-boundary-shapeder-scalar}.
\end{proof}

The following Lemma states that shape derivatives commute with linear transformations, both for domain and boundary functions. The proof is straightforward from the definitions.

\begin{lmm}\label{le:lineal-domain}
	Let $F\in \Lin(\V_1,\V_2)$, with $\V_1$ and $\V_2$ two finite dimensional vector or tensor spaces, and let $w(S)\in \mathcal{C}^k(S, \V_1)$ for any admissible domain or boundary 
	$S\subset \D$ and $k\geq 1$. If $w(S)$ is shape differentiable at $S$ in the direction $\bs{V}$, then $F\circ w(S) \in \mathcal{C}^k(S, \V_2)$ is also shape differentiable at $S$ in the direction $\bs{V}$, and its shape derivative is given by
	\begin{equation*}
	\left(F\circ w\right)'(S, \bs{V}) =  F\circ w'(S, \bs{V}).
	\end{equation*}
\end{lmm}

The next lemma states a chain rule combining usual derivatives with shape derivatives.


\begin{lmm}[Chain rule]\label{le:chain-rule}
	Consider two finite dimensional vector or tensor spaces $\V_1$ and $\V_2$, a function $F\in \mathcal{C}^1(\V_1,\V_2)$ and 
	a domain (or boundary) function $y(S) \in \mathcal{C}^1(S, \V_1)$, where $S$ is an admissible domain (boundary) in $\D \subset \R^N$ with a $\mathcal{C}^1$ boundary.  
	If $y(\Omega)$ is shape differentiable at $\Omega$ in the direction $\bs{V}$,	
	then the function $F\circ 
	y(\Omega)\in \mathcal{C}^1(\Omega, \V_2)$ is also shape differentiable at $\Omega$ 
	in the direction $\bs{V}$, and its shape derivative is given by
	\begin{equation*}
	\left(F\circ y\right)'(\Omega, \bs{V}) =DF\circ y(\Omega)\, [ y'(S, \bs{V})]. 
	\end{equation*}  
\end{lmm}

\begin{proof}
Since $F$ is differentiable, for every $X\in\V_1$ there exists a linear operator $DF(X) \in \Lin(\V_1,\V_2)$ such that
 	\begin{equation*}
 	\lim_{\|u\|_{\V_1}\rightarrow 0} \frac{\|F(X+u)-F(X) -DF(X)[u]\|_{\V_2}}{\|u\|_{\V_1}}=0;
 	\end{equation*}
 	where $DF(X)[u]$ denotes the application of the linear operator $DF(X)\in\Lin(\V_1,\V_2)$ to $u \in \V_1$.
 	With this notation, the chain rule applied to $F \circ y$ reads $D(F\circ y)[\bs{v}]=DF\circ y[Dy[\bs{v}]]$ ($\forall \bs{v} \in \R^N$), so that from Definition~\ref{def:domain-shape-der-vector},
	 \[
	 (F\circ y)'(\Omega, \bs{V}) =(F\circ y)^\cdotp(\Omega, \bs{V})- D(F\circ y)[\bs{V}]
	 = (F\circ y)^\cdotp(\Omega, \bs{V})- DF\circ y[Dy[\bs{V}]] , \quad \text{in $\Omega$},
	 \] with $(F\circ y)^\cdotp(\Omega, \bs{V})$ 
	 denoting the material derivative of $F\circ y$.
	 Then we only need to prove the chain rule for the material derivative of $y(\Omega)\in\mathcal{C}^1(\Omega, \V_1)$ 
	 in the direction $\bs{V}$, i.e.,
	 \[
	 (F\circ y)^\cdotp(\Omega, \bs{V}) = DF \circ y[\dot y(\Omega,\bs{V})]. 
	 \]
	 If we recall that $(F\circ y)^\cdotp(\Omega, \bs{V}) = \frac{d}{dt^+} \left[ F\circ y(\Omega_t)\circ x(t)\right]_{t=0}$,
	 the proof of this last equality is straightforward from usual chain rule applied to the mapping $t \to F\circ \left(y(\Omega_t)\circ x(t)\right)$. The details are left to the reader.
\end{proof}

\begin{rmrk}[Product rule for shape derivatives]\label{R:product shape derivatives}
The product rules for domain shape derivatives follow directly from 
Definitions \ref{def:function-domain-shape-der} and 
\ref{def:domain-shape-der-vector}.
\end{rmrk}   

The following Lemma allows us to swap shape derivatives with classical 
derivatives of domain functions. 
It is worth noting that this is not true for boundary functions and tangential derivatives. 
This will be discussed in Section~\ref{sec:main}, where the main results of 
this 
article are presented.

    \begin{lmm}[Mixed shape and classical derivatives]\label{le:conmute-der}
    The following results about interchanging classical and shape 
    derivatives are satisfied.
    \begin{enumerate}
    \item     
    If $y(\Omega)\in \mathcal{C}^2(\Omega)$  is shape differentiable at 
$\Omega$ in the direction 
    $\bs{V}\in  \mathcal{C}([0,\epsilon], V^k(\D)) $, $k\geq 2$, 
    then $\nabla y(\Omega)\in  \mathcal{C}^1(\Omega, \Rn[N])$ 
    is also shape differentiable at $\Omega$ and 
    \begin{equation*}
    (\nabla y)'(\Omega, \bs{V}) = \nabla y'(\Omega, \bs{V}) .
    \end{equation*}

    \item 
    If $\bs{w}(\Omega)\in \mathcal{C}^2(\Omega, \Rn[N])$  is shape 
differentiable at $\Omega$ in the
    direction $\bs{V}\in  \mathcal{C}([0,\epsilon], V^k(\D)) $, 
$k\geq 2$, 
    then $D \bs{w}(\Omega)\in  \mathcal{C}^1(\Omega, 
\mathbb{R}^{N\times N})$ and $\Div \bs{w}(\Omega)\in  \mathcal{C}^1(\Omega)$ 
    is also shape differentiable and 
    \begin{equation*}
    (D \bs{w})'(\Omega, \bs{V}) = D \bs{w}'(\Omega, \bs{V}),
    \qquad
    (\Div \bs{w})'(\Omega, \bs{V}) = \Div \bs{w}'(\Omega, \bs{V}). 
    \end{equation*}
    \item     
    If $y(\Omega)\in \mathcal{C}^3(\Omega)$  is shape differentiable 
at $\Omega$ 
    in the direction $\bs{V}\in  \mathcal{C}([0,\epsilon], V^k(\D)) 
$, $k\geq 3$, 
    then $\Delta y(\Omega)\in  \mathcal{C}^1(\Omega)$ 
    is also shape differentiable at $\Omega$ and 
    \begin{equation*}
    (\Delta y)'(\Omega, \bs{V}) = \Delta y'(\Omega, \bs{V}) .
    \end{equation*}

    \end{enumerate}
    \end{lmm}
    \begin{proof}
    We will prove the first assertion. The other ones are analogous.
    
    First note that
    \begin{equation*}
           \nabla \left(y(\Omega_t)\circ x(t)\right) =Dx(t)^T\, \nabla 
y(\Omega_t)\circ x(t) \ \ \ \text{ in } \Omega.
        \end{equation*}
Differentiating with respect to $t$ and evaluating at $t=0$ we have
     \begin{equation}\label{eq:sweping-derivatives}
      \frac{\partial}{\partial t} \nabla \left(y(\Omega_t)\circ x(t)\right) 
\, |_{t=0}= \frac{\partial}{\partial t}Dx(t)^T|_{t=0}\nabla y(\Omega) + 
      \dot{\nabla} y(\Omega, \bs{V})
     \end{equation}
where we have used that $x(0)=id$, $Dx(0)=I$ and Definition 
\ref{def:material-der}, denoting with $\dot{\nabla} y(\Omega, \bs{V})$
the material derivative of $\nabla y(\Omega)$.

Since  $x \in \mathcal{C}^1([0,\epsilon], 
V^k(\D))$, $k\geq 1$, and recalling that $\dot{x}(0) = 
\frac{\partial}{\partial 
t}x(t, \cdot)|_{t=0} =\bs{v}$,
we have that, uniformly
\begin{equation*}
\lim_{t\searrow 0}\,  D^\alpha\left(\frac{ x(t)- x(0)}{t}\right) 
=
D^\alpha \bs{v} \qquad \text{ for any } 0\leq |\alpha|\leq k,
\end{equation*}
so that
\begin{equation}\label{eq:swap-der-x}
\frac{\partial}{\partial t}Dx(t)|_{t=0} = D\bs{v} \qquad \text{ pointwise in } 
\D.
\end{equation}

Analogously, the existence of the material derivative $\dot{y}(\Omega, \bs{V})$ 
in $\mathcal{C}^1(\Omega)$
implies that, uniformly,
\begin{equation*}
\lim_{t\searrow 0}\frac{\partial}{\partial X_i}\, \left( \frac{ 
y(\Omega_t)\circ 
x(t)-y(\Omega)}{t}\right) =
\frac{\partial}{\partial X_i}\, \dot{y}(\Omega, \bs{V}), \qquad \text{ for } 
1\leq i \leq N,
\end{equation*} 
and then
\begin{equation}\label{eq:swap-der-y}
\frac{\partial}{\partial t} \nabla \left(y(\Omega_t)\circ x(t)\right) \, 
|_{t=0} 
= \nabla\dot{y}(\Omega, \bs{V})
\qquad \text{ pointwise in } \Omega.
\end{equation}

 Replacing with (\ref{eq:swap-der-x}) and (\ref{eq:swap-der-y}) in 
(\ref{eq:sweping-derivatives}), we obtain
\begin{equation*}
 \nabla\dot{y}(\Omega, \bs{V})= D\bs{v}^T\nabla y(\Omega)+ \dot{\nabla} 
y(\Omega, \bs{V}).
\end{equation*}

By Definition \ref{def:function-domain-shape-der} of shape derivative we have
\begin{align*}
 \nabla y'(\Omega, \bs{V}) &= \nabla\dot{y}(\Omega, \bs{V})-\nabla\left(\nabla 
y(\Omega)\cdot \bs{v}\right)
 =\nabla\dot{y}(\Omega, \bs{V})-D\bs{v}^T\nabla y(\Omega) - 
D^2y(\Omega)^T\bs{v}\\
 &= \dot{\nabla} y(\Omega, \bs{V}) - D^2y(\Omega)^T\bs{v}
 =\left(\nabla y\right)'(\Omega, \bs{V}),
\end{align*}
where we have used Definition \ref{def:domain-shape-der-vector} and the fact that $D^2 y(\Omega)^T$ is symmetric because $y(\Omega)\in \mathcal{C}^2(\Omega)$.
    \end{proof}

\subsection*{Shape derivative of the signed distance function}
Concerning the shape derivative of the signed distance function 
$b=b(\Omega)$, we cite \cite{DZ2011} and \cite{HR2003}.
Since $b|_\Gamma=0$, 
we can use Lemma~\ref{le:y_zero_Gamma} to obtain
\begin{equation}\label{eq:shape_der_of_b}
b'(\Omega, \bs{V})|_\Gamma = - \bs{v}\cdot \bs{n} = -v_n,
\end{equation}
where we have used that $\frac{\partial b}{\partial \bs{n}}=\nabla b \cdot 
\nabla b |_\Gamma = 1 $. Note that this is the shape derivative of 
$b(\Omega)$ only on $\Gamma$.

In order to find the shape derivative of $\bs{w}(\Omega)\definedas \nabla 
b(\Omega)$, we derive
the Eikonal Equation (\ref{eq:eikonal}) to obtain
$(\nabla b )'(\Omega, \bs{V}) \cdot \nabla b = 0$.
This  means that, when restricted to $\Gamma$, $(\nabla b )'(\Omega, \bs{V})$ 
is 
orthogonal to the normal vector field $\bs{n}$. But also from Lemma 
\ref{le:conmute-der} $(\nabla b )'(\Omega, \bs{V})  =  \nabla\, b'(\Omega, 
\bs{V})$  so that
$
(\nabla b )'(\Omega, \bs{V})|_\Gamma  
 =  \gradg (b'(\Omega, \bs{V}))
 =  \gradg b'(\Omega, \bs{V}) .
$

Finally, using formula (\ref{eq:shape_der_of_b}), we obtain
\begin{equation}\label{eq:nabla-prima}
(\nabla b )'(\Omega, \bs{V})|_\Gamma = - \gradg v_n, \ \ \text{ where }v_n= 
\bs{v}\cdot \bs{n}.
\end{equation}
This identity will be used later to obtain the shape derivative of the boundary (vector valued) function $\bs{n}$.

    \section{Properties of the shape derivative of boundary 
functions}\label{S:boundary-properties}
    
We first extend Definition \ref{def:shape-der-boundary-scalar} to vector and tensor valued functions defined on a boundary $\Gamma=\partial \Omega$. To consider shape derivatives in a fixed surface $\Gamma \subsetneq \partial \Omega$, we consider the space of velocities $V_\Gamma(\D)$ instead of $V(\D)$ as it was done in Remark \ref{rem:surfaces-velocities}.
     
 \begin{dfntn}[Vector Boundary Functions]\label{def:boundary-shape-der-vector}

       Consider a vector field $w(\Gamma)\in W^{r,p}(\Gamma, \Rn[N])$ which 
exists for all admissible domains 
       $\Omega\subset \D$ with boundary $\Gamma \in \mathcal{C}^k$, with $1\leq 
r \leq k$.
   For a given velocity field  $\bs{V}\in  \mathcal{C}([0,\epsilon], V^k(\D)) 
$, 
we say that $\bs{w}(\Gamma)$ 
   is shape differentiable in $\Gamma$ in the direction of $\bs{V}$ if there 
exists the material derivative 
   $\dot{\bs{w}}(\Gamma, \bs{V}) = \frac{d}{dt} [\bs{w}(\Gamma_t)\circ 
x(t)]_{t=0}$ in $W^{r-1,p}(\Gamma, \Rn[N])$.
   In that case the (boundary) shape derivative belongs to
   $W^{r-1,p}(\Gamma, \Rn[N])$ and is given by
   \begin{equation*}
    \bs{w}'(\Gamma, \bs{V})= \dot{\bs{w}}(\Omega, \bs{V}) - \Dg\bs{w}\bs{v},
   \end{equation*}
where $\bs{v}=\bs{V(0)}$.

Analogously to tensor valued domain functions,
for a tensor valued boundary function $A(\Gamma):\Gamma \rightarrow \mathbb{R}^{N\times 
N}$ the shape derivative $A'(\Gamma, \bs{V})$ is the tensor valued function which satisfies (\ref{eq:def-tensorshape}), with $\Omega$ 
replaced by $\Gamma$.
\end{dfntn}

    We now extend Lemma \ref{le:domain-boundary-shapeder-scalar} to vector 
valued functions. The proof is straightforward from the fact that 
$\bs{w}(\Omega,\bs{V})\cdot \bs{e} = (\bs{w}\cdot\bs{e})'(\Omega,\bs{V})$ for 
any fixed vector $\bs{e}$, because $\bs{w} \to \bs{w}\cdot \bs{e}$ is a linear 
transformation.
    
    \begin{lmm}[Extension]\label{le:domain-boundary-shapeder-vector}
    Let $\Omega$ be a domain with a $\mathcal{C}^1$ boundary $\Gamma$. If 
the boundary function $\bs{w}(\Gamma) \in \mathcal{C}^1(\Gamma, \Rn[N])$ admits 
an extension $\bs{W}(\Omega)$ which is shape differentiable at $\Omega$ in the 
direction $\bs{V}\in  \mathcal{C}([0,\epsilon], V^k(\D)) $, $k\geq 1$, then the 
shape derivative of $\bs{w}(\Gamma)$ in $\Gamma$ at $\bs{V}$ is given by
    \begin{equation*}
    \bs{w}'(\Gamma, \bs{V})= \bs{W}'(\Omega, \bs{V})|_\Gamma + 
D\bs{W}\bs{n}\, v_n.
    \end{equation*} 
    \end{lmm}

    \begin{rmrk}[Comparison with Lemma \ref{le:properties_domain}]
    The first two assertions of Lemma \ref{le:properties_domain} are also valid 
for boundary functions (the proofs can be found in \cite{SZ1992}). Instead of 
the third assertion, we have the following one, which is an immediate consequence of Lemma~\ref{le:domain-boundary-shapeder-scalar}.
    
    \begin{lmm}[Shape derivative of $\phi|_\Gamma$]
       If $z(\Gamma)=\phi|_\Gamma$, for $\phi\in W^{r,p}(\D)$, 
$r-\frac{1}{p}\geq 1$, then $z$ is shape differentiable in $\Gamma$ for any 
direction $\bs{V}$, and $y'(\Gamma, \bs{V})=\frac{\partial \phi}{\partial n} 
v_n.$
      \end{lmm}
    
    \end{rmrk}
    \subsection*{The shape derivatives of $\bs{n}$ and $\kappa$}
    
   Since the normal vector field $\bs{n}$ at $\Gamma$ has the gradient of the 
signed distance function $b(\Omega)$ as an extension, we can use the results of 
   Section 7.1 to obtain the shape derivative of $\bs{n}$.
    
    \begin{lmm}[Shape derivative of $\n$]\label{le:normal_shape_der}
    Let $\Gamma$ be the boundary of a $C^2$-domain $\Omega$ and let 
$\bs{n}(\Gamma)$ be the unit normal vector field  of $\Gamma$. Then $\bs{n}$ is 
shape differentiable and
    $\bs{n}'(\Gamma, \bs{V}) = -\gradg v_n$,
    where $v_n = \bs{V}(0) \cdot \bs{n}$, for all $\bs{V} \in 
\mathcal{C}([0,\epsilon], V^1(\D))$.
    \end{lmm}
    
   \begin{proof}
    Since $\bs{n}(\Gamma)= \nabla b |_\Gamma$,
     where $b$ is the signed distance function of $\Omega$, we use Lemma 
\ref{le:domain-boundary-shapeder-vector} to obtain
     \begin{equation*}
      \bs{n}'(\Gamma, \bs{V})=(\nabla b)'(\Omega, \bs{V})|_\Gamma+ D^2b|_\Gamma 
\bs{n}\, v_n.
     \end{equation*}
Since by equation (\ref{eq:eikonal2})  $D^2b \nabla b= 0$ everywhere, the 
second 
term vanishes and~(\ref{eq:nabla-prima}) yields the desired result.
   \end{proof}

 Concerning  the mean curvature $\kappa =\kappa(\Gamma)$ given by 
(\ref{curvatures}),
 we recall that $\Delta b (\Omega)=\tr (D^2 b)$ is an extension of $\kappa$ to 
a tubular neighborhood of $\Gamma$ (see Section~\ref{S:signed distance function}).
  Then, on the one hand, the shape derivative of $\kappa$ can be obtained using Lemma 
\ref{le:domain-boundary-shapeder-scalar}, as we will do in the proof of the 
next Lemma.
On the other hand, since $D^2b|_\Gamma =\Dg\bs{n}$, we can also express the mean 
curvature using tangential derivatives, as $\kappa = \Divg\bs{n}$.
Then, Corollary \ref{cor:main} gives us quickly the same formula for the shape 
derivative of $\kappa$.

\begin{lmm}[Shape derivative of $\kappa$]\label{le:curvature_shape-der}
 If $\kappa$ is the mean curvature of  $\Gamma$, the boundary of a 
$\mathcal{C}^3$ domain $\Omega$, then $\kappa$ is shape differentiable in
 $\Gamma$ and
 \begin{equation*}
  \kappa'(\Gamma, \bs{V}) =-\lb v_n - |\Dg\bs{n}|^2 v_n,
 \end{equation*}
where $v_n=\bs{V}(0)\cdot \bs{n}$, 
$|\Dg\bs{n}|^2=\Dg\bs{n}:\Dg\bs{n}=tr(\Dg\bs{n}^2)$ and $\lb f = \Divg \gradg 
f$ 
is the Laplace-Beltrami operator of $f$. 
\end{lmm}

 \begin{proof}
 Since $\kappa =\Delta b|_\Gamma$, by Lemma 
\ref{le:domain-boundary-shapeder-scalar}  
$\kappa'(\Gamma, \bs{V})= (\Delta b)'(\Omega, \bs{V})|_\Gamma + (\nabla 
\Delta b \cdot \nabla b )|_\Gamma v_n$.
The second term is equal to $-|\Dg\bs{n}|^2 v_n$ using equation 
(\ref{eq:eikonal3}), and the fact that $\Dg\bs{n}=D^2 b |_\Gamma$.
For the first term, we use Lemma \ref{le:conmute-der} and the definition of 
tangential divergence~(\ref{eq:tan-der-def3}) to obtain
\begin{align*}
 (\Delta b)'(\Omega, \bs{V})&=\Div (\nabla b)'(\Omega, \bs{V})
                            = \Divg (\nabla b)'(\Omega, \bs{V}) + 
D^2b|_\Gamma\bs{n}\cdot \bs{n}\\
                            &=-\Divg(\gradg v_n)
                            = -\lb v_n,
\end{align*}
where we have used (\ref{eq:nabla-prima}) and the fact that $D^2b|_\Gamma\bs{n} 
= D^2b|_\Gamma \nabla b|_\Gamma = (D^2b \nabla b)|_\Gamma = 0$. 
 \end{proof}

\section{The shape derivatives of tangential operators}\label{sec:main}
We are now in position to present the main results of this paper, namely,
formulas for the shape derivatives of boundary functions that are
tangential derivatives of boundary functions. 
More precisely, we will find the shape derivatives of boundary functions of 
the form 
$\gradg z$, $\Dg\bs{w}$, $\Divg \bs{w}$ and $\lb z$, when $z(\Gamma)$ and 
$\bs{w}(\Gamma)$ 
are shape differentiable boundary functions, scalar and vector
  valued, respectively.
  Examples of important applications will be presented in the two subsequent 
sections. 

It is worth noting the difference with Lemma~\ref{le:conmute-der} where we established that standard differential operators commute with the shape derivative of domain functions.
    
    \begin{thrm}[Shape derivative of surface derivatives]\label{thm:main} 
    
    For any admissible boundary $\Gamma = \partial \Omega$, where $\Omega$ 
is a $\mathcal{C}^2$ domain in
    $\D\subset\Rn[N]$, consider a real function 
$z(\Gamma)\in\mathcal{C}^2(\Gamma)$. 
    If $z$ is shape differentiable at $\Gamma$ in the direction 
$\bs{V}\in\mathcal{C}([0,\epsilon], V^2(\D))$, then $ \gradg z$ is shape 
differentiable at $\Gamma$ in the direction $\bs{V}$, and
    \begin{equation}\label{eq:theorem1}
    \left(\gradg z \right)'(\Gamma, \bs{V}) = \gradg z'(\Gamma, \bs{V}) + 
\left( \bs{n} \tensor \gradg v_n- v_n\Dg \bs{n} \right) \gradg z,
    \end{equation}
    where $v_n= \bs{V}(0)\cdot \bs{n}$.
    \end{thrm}
    
    Before proceeding to the proof it is worth noticing the differences with 
Lemma~\ref{le:conmute-der} where we have shown that the shape derivative of 
domain integrands commutes with the space derivatives.

        \begin{proof}
        Let $y=y(\Omega)$ be an extension of $z(\Gamma)$ to $\Omega$, i.e. $z(\Gamma)=y(\Omega)|_\Gamma$, then by definition 
    \begin{equation*}
    \gradg z(\Gamma) = \nabla y|_\Gamma- \partial_{n} y  
    \bs{n} = \left( \nabla y - (\nabla y \cdot \nabla b) \nabla b \right) 
|_\Gamma,
    \end{equation*}
    because $\partial_{n} y  = \frac{\partial y}{\partial n} =\nabla y \cdot 
\bs{n}$.
    Then $\bs{\Phi} (\Omega) \definedas \nabla y - (\nabla y \cdot \nabla b) 
\nabla b$ is an extension to $\Omega$ of
    $\gradg z (\Gamma)$. Due to Lemma~\ref{le:domain-boundary-shapeder-vector} 
these shape derivatives satisfy
    \begin{equation}\label{eq:related_derivative}
    \left(\gradg z \right)'(\Gamma, \bs{V})= \bs{\Phi}'(\Omega, 
\bs{V})|_\Gamma +D\bs{\Phi}(\Omega)|_\Gamma\bs{n}\,v_n.
    \end{equation}
    
    We now compute the domain shape derivative of $\bs{\Phi}(\Omega)$.
    Using the product rule we have
    \begin{align*}
    \bs{\Phi}'(\Omega, \bs{V})& = (\nabla y)'(\Omega, \bs{V})- (\nabla 
y)'(\Omega, \bs{V})\cdot \nabla b\ \nabla b- 
    \nabla y \cdot (\nabla b)'(\Omega, \bs{V})\ \nabla b - \nabla y \cdot 
\nabla b \ (\nabla b)'(\Omega, \bs{V})\\
    &= \left(I-\nabla b\tensor \nabla b\right) \nabla y'(\Omega, 
\bs{V})-\nabla y \cdot (\nabla b)'(\Omega, \bs{V})\ \nabla b - \nabla y \cdot 
\nabla b \ (\nabla b)'(\Omega, \bs{V}),
    \end{align*}
    where we have used Lemma~\ref{le:conmute-der} to commute the shape 
derivative and the gradient of $y$.
Restricting to $\Gamma$, using the definition of tangential gradient and 
formula 
(\ref{eq:nabla-prima}) for $(\nabla b)'(\Omega, \bs{V})|_\Gamma$, we obtain     
 \begin{equation*}
    \bs{\Phi}'(\Omega, \bs{V})|_\Gamma 
     = \gradg y'(\Omega, \bs{V}) + (\bs{n}\tensor \gradg v_n) \gradg z + 
\partial_n y \gradg v_n,
    \end{equation*}
    where we have used
    $\nabla y(\Omega)|_\Gamma \cdot \gradg v_n\ \bs{n} = \gradg z \cdot 
\gradg v_n\ \bs{n} = (\bs{n}\tensor \gradg v_n) \gradg z$.
        From Lemma~\ref{le:domain-boundary-shapeder-scalar} $y'(\Omega, 
\bs{V})|_\Gamma = z'(\Gamma, \bs{V})-\partial_n y\, v_n$ 
    and the product rule for tangential derivative yields
    \begin{align*}
    \bs{\Phi}'(\Omega, \bs{V})|_\Gamma & =  \gradg z'(\Gamma, \bs{V}) - 
\gradg(\partial_n y\, v_n)  
    +(\bs{n}\tensor \gradg v_n) \gradg z + \partial_n y \gradg v_n\\ 
    & =   \gradg z'(\Gamma, \bs{V}) + (\bs{n}\tensor \gradg v_n) \gradg z - 
v_n\gradg (\partial_n y).
    \end{align*}
    
    Then, from (\ref{eq:related_derivative}), to complete the proof of 
(\ref{eq:theorem1}), we need to show that 
    \begin{equation}\label{eq:final_eq}
    D\bs{\Phi}(\Omega)|_\Gamma\bs{n} -  \gradg (\partial_n y) = - \Dg \bs{n} 
\gradg z.
    \end{equation}
    
    Applying the product rule of classical derivatives to $\bs{\Phi} 
(\Omega) = \nabla y - (\nabla y \cdot \nabla b) \nabla b$, we obtain, using $\n 
= \nabla b|_\Gamma$,
    \begin{align*}
    D\bs{\Phi}(\Omega)|_\Gamma\bs{n} &= D^2y|_\Gamma \bs{n} - 
    \left( \bs{n} \tensor \nabla(\nabla y \cdot \nabla b)|_\Gamma \right) 
\bs{n} - \partial_n y D^2b \nabla b |_\Gamma\\  
    &=  D^2y|_\Gamma \bs{n} - \partial_n(\nabla y \cdot \nabla b)\bs{n},
    \end{align*}
    because $D^2b \nabla b=0$.
    Besides, 
    \begin{align*}
    \gradg(\partial_n y) &= \nabla(\nabla y \cdot \nabla b)|_\Gamma - 
\partial_n (\nabla y \cdot \nabla b) \bs{n} \\
    &= D^2y|_\Gamma \bs{n} - D^2b \nabla y|_\Gamma - \partial_n (\nabla y 
\cdot \nabla b) \bs{n} \\
    &= D\bs{\Phi}(\Omega)|_\Gamma\bs{n} - \Dg\bs{n} \gradg z,
    \end{align*}
    where we have used that
    $D^2b \nabla y|_\Gamma = \Dg\bs{n} \gradg y
    = \Dg\bs{n} \gradg z$.
    From this equation we obtain (\ref{eq:final_eq}) and the claim follows.
    \end{proof}

    \begin{crllr}[For vector fields]\label{cor:main}
    If the functions $z(\Gamma)\in\mathcal{C}^2(\Gamma)$ and 
$\bs{w}(\Gamma)\in\mathcal{C}^2(\Gamma,\Rn[N])$ are shape differentiable at 
$\Gamma\in \mathcal{C}^2 $ in the direction  
$\bs{V}\in\mathcal{C}([0,\epsilon], 
V^2(\D))$, then $\Dg \bs{w}$ and $\Divg\bs{w}$ are also shape differentiable at 
$\Gamma$ in the direction $V$ and 
    \begin{align}\label{eq:shape-der-Dg}
    (\Dg \bs{w})'(\Gamma, \bs{V}) &= \Dg \bs{w}'(\Gamma, \bs{V}) + 
\Dg\bs{w}[\gradg v_n\tensor \bs{n} - v_n\Dg \bs{n}], \\
\label{eq:shape-der-Divg}
    (\Divg\bs{w})'(\Gamma, \bs{V}) &= \Divg \bs{w}'(\Gamma, \bs{V}) +  
[\bs{n}\tensor\gradg v_n - v_n\Dg \bs{n}]:\Dg\bs{w},
    \end{align}
    where $S:T=\tr(S^T T)$ denote the scalar product of tensors.
    \end{crllr}
    
    \begin{proof}
    In order to obtain (\ref{eq:shape-der-Dg}), note that 
$\Dg\bs{w}^T\bs{e}_i=\gradg w_i$, where $w_i=\bs{w}\cdot \bs{e}_i$, with 
$\{\bs{e}_1, ..., \bs{e}_N\}$ being the canonical basis of $\Rn[N]$. By 
definition, the shape derivative of the tensor $\Dg\bs{w}^T$ must satisfy 
    \begin{equation*}
    (\Dg\bs{w}^T)'(\Gamma, \bs{V})\bs{e}_i= (\Dg\bs{w}^T\bs{e}_i)'(\Gamma, 
\bs{V})= (\gradg w_i)'(\Gamma, \bs{V}). 
    \end{equation*}
    Applying (\ref{eq:theorem1}) to $z(\Gamma)=w_i=\bs{w}\cdot 
\bs{e}_i$, we obtain
    \begin{align*}
    (\Dg \bs{w}^T)'(\Gamma, \bs{V})\bs{e}_i &= (\gradg w_i)'(\Gamma, \bs{V})\\
    &= \gradg w_i'(\Gamma, \bs{V}) + [\bs{n}\tensor\gradg v_n - v_n\Dg 
\bs{n}]\gradg w_i\\    
    &=
    \left(\Dg \bs{w}'(\Gamma, \bs{V})^T + [\bs{n}\tensor\gradg v_n  - v_n\Dg 
\bs{n}]\Dg\bs{w}^T\right)\bs{e}_i. 
    \end{align*}
    The linearity of the transpose operator and 
Lemma~\ref{le:lineal-domain} yield the desired result.
    
    Finally, we recall that $(\Divg\bs{w})'(\Gamma, \bs{V})=\tr(\Dg 
\bs{w})'(\Gamma, \bs{V})$ and
$(\bs{a}\tensor \bs{b}): S = \bs{a}\cdot S \bs{b}$.
Therefore~(\ref{eq:shape-der-Dg}) implies
    \begin{equation*} 
    (\Divg\bs{w})'(\Gamma, \bs{V})= \Divg \bs{w}'(\Gamma, \bs{V}) + 
    \Dg\bs{w}\gradg v_n \cdot \bs{n}- v_n\Dg \bs{n}:\Dg\bs{w},
    \end{equation*}
and~(\ref{eq:shape-der-Divg}) follows.
    \end{proof}
    
    We end this section establishing the shape derivative of the 
Laplace-Beltrami operator of a boundary function, which is more involved 
because 
it is of second order.
    
    \begin{thrm}[Shape derivative of Lapace-Beltrami]\label{thm:shape-der-lb}
    If $z=z(\Gamma)\in \mathcal{C}^3(\Gamma)$ is shape differentiable at a 
$\mathcal{C}^3$-boundary $\Gamma$ in the direction 
$\bs{V}\in\mathcal{C}([0,\epsilon], V^3(\D))$,  then $\lb z \definedas \Divg 
\gradg z$ 
    is also shape differentiable at $\Gamma$ in the direction $\bs{V}$, and 
its shape derivative is given by
    \begin{equation}\label{eq:shape-der-laplace-beltrami}
    \begin{split}
    (\lb z)'(\Gamma, \bs{V})
    &= \lb z'(\Gamma, \bs{V}) -2v_n\Dg\bs{n}:\Dg^2z
    +\left( \kappa \gradg v_n- 2\Dg\bs{n}\gradg v_n- v_n\gradg \kappa 
\right)\cdot \gradg z
    \\
    &= \lb z'(\Gamma, \bs{V}) 
        -v_n\left(2\Dg\bs{n}:\Dg^2z +\gradg \kappa \cdot \gradg z 
\right)
        +  \gradg v_n \cdot \left( \kappa \gradg z - 2\Dg\bs{n}\gradg 
z\right) .
    \end{split}
    \end{equation}
    \end{thrm}
    
    \begin{proof}
     In order to simplify the calculation, we denote $M = \bs{n}\tensor\gradg 
v_n-v_n\Dg\bs{n}$.
     Using successively the formulas for the shape derivative of a tangential 
divergence (Corollary \ref{cor:main}) and for a tangential gradient (Theorem 
\ref{thm:main}), we have
      \begin{align*}
      (\lb z)'(\Gamma, \bs{V})
      &= (\Divg \gradg z)'(\Gamma, \bs{V}) \\
      &= \Divg ( (\gradg z)'(\Gamma,V)) + M: \Dg \gradg z \\
       &= \Divg [\gradg z'(\Gamma, \bs{V})+ M\gradg z] + M:\Dg^2z \\
       &=\lb z'(\Gamma, \bs{V})+\Divg ( M\gradg z)+ M:\Dg^2z.
       \end{align*}
       Using the product rule (v) of Lemma \ref{le:DgProperties} we obtain
        \begin{align}
       (\lb z)'(\Gamma, \bs{V})&= \lb z'(\Gamma, \bs{V})+M^T:\Dg^2z+\Divg 
M^T\cdot\gradg z+ M:\Dg^2z \nonumber \\ 
       &= \lb z'(\Gamma, \bs{V})+(M+M^T):\Dg^2z+\Divg M^T\cdot \gradg z. \label{eq:a0}
      \end{align}
      
     Since $D_\Gamma \n^T = D_\Gamma \n$, the second term of the right-hand side reads
      \begin{equation*}
       M+M^T= \bs{n}\tensor \gradg v_n+ \gradg v_n \tensor \bs{n} -2 v_n 
\Dg\bs{n}.
      \end{equation*}
      Using the tensor property $(\bs{a}\tensor \bs{b}): S = \bs{a}\cdot S 
\bs{b}$ and that $\Dg^2 z \,\bs{n}=0$, we obtain
      \begin{equation*}
       (M+M^T):\Dg^2z =\bs{n}\cdot \Dg^2z \gradg v_n - 2 v_n\Dg\bs{n}: \Dg^2 z.
      \end{equation*}
Observe that differentiating $\bs{n}\cdot \gradg z=0$ leads to 
$\Dg^2z^T\bs{n}=-\Dg\bs{n}\gradg z$ which implies 
$\bs{n}\cdot \Dg^2z\gradg v_n = -\Dg\bs{n}\gradg v_n\cdot \gradg z$. Then
      \begin{equation}\label{eq:a1}
       (M+M^T):\Dg^2z =-\Dg\bs{n}\gradg v_n\cdot \gradg z - 2 v_n\Dg\bs{n}: 
\Dg^2 z.
      \end{equation}

     The third term of~(\ref{eq:a0}) contains $\Divg M^T$ which can be computed with the product rules of Lemma~\ref{le:DgProperties} to obtain
     \begin{align*}
      \Divg M^T &=\Divg (\gradg v_n \tensor \bs{n})-\Divg(v_n \Dg\bs{n})\\
      &= \gradg v_n \cdot \Divg \n + \Dg \gradg v_n \n - \Dg \n \gradg v_n - 
v_n 
\Divg(\Dg\n)\\
      &=\kappa \gradg v_n - \Dg\bs{n}\gradg v_n - v_n \lb\bs{n},
     \end{align*}
where we have used that $\kappa=\Divg\bs{n}$ and $\Dg \gradg v_n \n = 
\Dg^2v_n\bs{n}=0$. Since $\lb\bs{n}\cdot\gradg z = (P\lb\bs{n})\cdot\gradg z$,
where $P$ is the orthogonal projection to the tangent plane,
equation (\ref{eq:grad_curvature}) yields $(P\lb\bs{n})=\gradg \kappa$ whence
\begin{equation}\label{eq:a2}
 \Divg M^T\cdot \gradg z =\kappa \gradg v_n\cdot \gradg z- \Dg\bs{n} \gradg v_n 
\cdot \gradg z-v_n \gradg \kappa \cdot \gradg z.
\end{equation}

Finally we add equations (\ref{eq:a1}) and (\ref{eq:a2}) and replace in 
(\ref{eq:a0}) to obtain 
    \begin{equation*}
     (\lb z)'(\Gamma, \bs{V})= \lb z'(\Gamma, \bs{V})-2\Dg\bs{n}\gradg v_n\cdot 
\gradg z - 2 v_n\Dg\bs{n}: \Dg^2 z+
     \kappa \gradg v_n\cdot \gradg z-v_n \gradg \kappa \cdot \gradg z,
    \end{equation*}
which completes the proof.
    \end{proof}

\begin{crllr}[Shape derivative of the first fundamental 
form]\label{cor:shapeder_curvature_Dn}
 For a $\mathcal{C}^3$ surface $\Gamma$ and a smooth velocity field $\bs{V}$, 
the shape derivatives of the mean curvature $\kappa=\Divg\bs{n}$ and the tensor 
$\Dg\bs{n}$ at $\bs{V}$ are given by
 \begin{equation}\label{eq:curvature_shapeder}
 \kappa'(\Gamma, \bs{V}) =-\lb v_n - |\Dg\bs{n}|^2 v_n,
 \end{equation}
    and 
 \begin{equation}\label{eq:derDgamman}
 (\Dg\bs{n})'(\Gamma, \bs{V}) = -\Dg^2v_n+ \Dg\bs{n}\gradg v_n\tensor \bs{n} 
-v_n\Dg\bs{n}^2.\end{equation}   
    \end{crllr}

    \begin{proof} 
Since, by Lemma \ref{le:normal_shape_der}, $\bs{n}'(\Gamma, \bs{V}) = -\gradg 
v_n$, we will use corollary \ref{cor:main} to obtain the shape derivatives of 
$\Dg\bs{n}$ and $\kappa=\Divg\bs{n}$. From (\ref{eq:shape-der-Dg}) we have
\begin{align*}
(\Dg \bs{n})'(\Gamma, \bs{V}) &= \Dg \bs{n}'(\Gamma, \bs{V}) + 
\Dg\bs{n}[\gradg v_n\tensor \bs{n} - v_n\Dg \bs{n}]\\
&=-\Dg^2v_n+ \Dg\bs{n}\gradg v_n\tensor \bs{n} -v_n\Dg\bs{n}^2,
    \end{align*}
    whence (\ref{eq:derDgamman}) holds.
    To obtain the shape derivative of $\kappa=\Divg\bs{n}$ we can use equation 
(\ref{eq:shape-der-Divg}) or observe that $(\Divg\bs{n})'(\Gamma, 
\bs{V})=\tr(\Dg 
\bs{n})'(\Gamma, \bs{V})$
    and use (\ref{eq:derDgamman}). In both cases, note that
    $\Divg \gradg v_n = \tr(\Dg^2 v_n)= \lb v_n $ and $\Dg\bs{n}:\Dg\bs{n} = 
\tr(\Dg\bs{n}^2)= |\Dg\bs{n}|^2$. Also, since $\Dg \bs{n} \bs{n} =0$, we have
    \[
    \tr(D_\Gamma \n \gradg v_n \tensor \n) = \bs{n}\tensor\gradg v_n :\Dg\bs{n} = \Dg\bs{n}\gradg v_n \cdot \bs{n} 
=0.
\]
 We have thus obtained $(\Divg\bs{n})'(\Gamma, \bs{V})= \kappa'(\Gamma, \bs{V})= -\lb v_n - 
|\Dg\bs{n}|^2 v_n$.
    \end{proof}
  
    \section{Geometric invariants and Gaussian curvature}\label{sec:invariants}
    
   The geometric invariants of a $\mathcal{C}^2$-surface $\Gamma$ allow us to define its 
intrinsic properties. 
They are defined as the geometric invariants of the tensor $\Dg \n$, which, in turn, are the coefficients of its characteristic polynomial $p(\lambda)$ (see \cite{Li2007}).
The geometric invariants of $\Gamma$, $i_j(\Gamma): \Gamma 
\rightarrow \R$, $j= 1,\dots, N$,  thus satisfy
   \[
   p(\lambda)=\det(\Dg\bs{n}(X)-\lambda I)=\lambda^N + i_1 \lambda^{N-1}+
   i_2 \lambda^{N-2}+ ... + i_{N-1}\lambda + i_N,
   \]
and can also be expressed using the eigenvalues of the 
tensor $\Dg \bs{n}$, which  are zero and the principal curvatures 
$\kappa_1,\dots,\kappa_{N-1}$, namely
\[i_1(\Gamma) = \sum_{j=1}^{N-1}\kappa_j,\quad
i_2(\Gamma) = \sum_{j_1\neq j_2}\kappa_{j_1}\kappa_{j_2},\quad 
\dots,\quad
i_{N-1}(\Gamma)= \kappa_1 \dots \kappa_{N-1},\quad
i_N(\Gamma)=0.
\]
   
   We can observe from definitions (\ref{curvatures}) that the first invariant 
$i_1(\Gamma)$ is the \emph{mean curvature} $\kappa$ and 
   the last nonzero invariant $i_{N-1}(\Gamma)$ is the \emph{Gaussian 
curvature} 
$\kappa_g$.  
   The invariant $i_k(\Gamma)$, for $2\leq k \leq N-2$, is the so-called $k$-$th$ 
\emph{mean curvature}~\cite[Ch.~3F]{Kuhnel2013}.
   
The geometric invariants of $\Gamma$ can also be defined through the 
functions
$I_p(\Gamma):\Gamma \rightarrow \R$, given by $I_p(\Gamma)= 
\tr(\Dg\bs{n}^p)$, $p=1, ..., N-1$. 
In particular, the first 4 invariants are
  $$i_1 = I_1=\Divg \bs{n},$$
  $$i_2 = \frac{1}{2!}\left(I_1^2-I_2\right)= 
\frac{1}{2}\left(\kappa^2-|\Dg\bs{n}|^2\right),$$
  $$i_3 = \frac{1}{3!}\left(I_1^3-3I_1I_2+2I_3\right),$$
  $$i_4 = \frac{1}{4!}\left(I_1^4-6I_1^2I_2+3I_2^2+8 I_1I_3 -6I_4\right).$$

   We will now establish the shape derivatives of the functions 
$I_p(\Gamma)=\tr(\Dg \n^p)$, which are also intrinsic to the surface $\Gamma$ 
and will lead to the shape derivatives of the geometric invariants 
$i_k(\Gamma)$. 
   
   \begin{prpstn}[Shape derivatives of the invariants]\label{prop:derIp}
   Let $\Gamma$ be a $\mathcal{C}^2$-boundary in $\R^N$ and $p$ a positive 
integer.
   The shape derivative of the scalar valued boundary function 
$I_p(\Gamma)\definedas \tr(\Dg\bs{n}^p)$ at $\Gamma$ in the direction 
$\bs{V}\in\mathcal{C}([0,\epsilon], V^2(\D))$ is given by
   \begin{equation*}\label{eq:prop-derIp}
   (I_p)'(\Gamma, \bs{V}) = -p\left(\Dg^2v_n : \Dg\bs{n}^{p-1}+v_n I_{p+1} 
\right),
   \end{equation*}
where $v_n=\bs{V(0)}\cdot \bs{n}$.
   \end{prpstn}
   
   For the proof of this proposition we will need the following Lemma.
   \begin{lmm}\label{le:simetric-tensor}
   Let $A(\Gamma):\Gamma \rightarrow \Lin(\mathbb{V})$ be a symmetric 
tensor valued function and let $p$ be a positive integer. If $A(\Gamma)$ is 
shape differentiable at $\Gamma$ in the direction $\bs{V}$, then
   the shape derivative of $A^p(\Gamma)$ satisfies
   \begin{equation}\label{eq:simetric-tensor}
   (A^p)'(\Gamma, \bs{V}):A^j =p\left(A'(\Gamma, \bs{V}):A^{j+p-1}\right),
   \end{equation}
   for any integer $j\geq 0$.
   \end{lmm}
   \begin{proof}
   We will proceed by induction. It is trivial to see that equation 
(\ref{eq:simetric-tensor}) holds for $p=1$ and any integer $j\geq 0$.
   
   Assuming that  equation (\ref{eq:simetric-tensor}) holds for 
$p\geq 1$ and any $j\geq 0$, we want to prove
   \begin{equation}\label{eq:simetric-tensor2}
   (A^{p+1})'(\Gamma, \bs{V}):A^j =(p+1)\left(A'(\Gamma, 
\bs{V}):A^{j+p}\right), \ \ \ \text{ for any integer } j\geq 0.
   \end{equation}
   Applying the product rule for the shape derivative to 
$A^{p+1}=A^pA$, we have 
   \begin{equation*}
   (A^{p+1})'(\Gamma, \bs{V}):A^j = (A^{p})'(\Gamma, \bs{V})A:A^j 
+A^p A'(\Gamma, \bs{V}):A^j.
   \end{equation*}
   The tensor product property $BC:D=B:DC^T = C:B^TD$ and the fact 
that the tensor $A$ is symmetric, yield
   \begin{equation}\label{eq:simetric-tensor3}
   (A^{p+1})'(\Gamma, \bs{V}):A^j =(A^{p})'(\Gamma, 
\bs{V}):A^{j+1}+
   A'(\Gamma, \bs{V}):A^{j+p}.
   \end{equation}
    
    The inductive assumption for $p$ and $j+1$ implies 
    $(A^p)'(\Gamma, \bs{V}):A^{j+1} =p\left(A'(\Gamma, 
\bs{V}):A^{j+p}\right)$. Using this in equation (\ref{eq:simetric-tensor3}), we 
obtain the desired result (\ref{eq:simetric-tensor2}).
   \end{proof}
   
   \begin{proof}[Proof of Proposition \ref{prop:derIp}]
   First note that $I_p'(\Gamma,\bs{V})= \tr(\Dg\bs{n}^p)'(\Gamma, \bs{V})= 
(\Dg\bs{n}^p)'(\Gamma, \bs{V}) : \Dg\bs{n}^0$. Then Lemma~\ref{le:simetric-tensor} with $j=0$ and $A=\Dg\bs{n}$, which is a symmetric tensor, lead to
   \begin{equation*}
   I_p'(\Gamma, \bs{V})= p\left(\Dg\bs{n}'(\Gamma, 
\bs{V}):\Dg\bs{n}^{p-1}\right).
   \end{equation*} 
   
   From formula (\ref{eq:derDgamman}) we have that
   $(\Dg\bs{n})'(\Gamma, \bs{n}):\Dg\bs{n}^{p-1} = -\Dg^2v_n: 
\Dg\bs{n}^{p-1}  - v_n I_{p+1}(\Gamma)$,
where we have used that 
$\Dg\bs{n}\gradg v_n \tensor \bs{n} 
:\Dg\bs{n}^{p-1}=0$ for any integer $p\geq 1$. This completes the proof.
   \end{proof}
   
   We now obtain the shape derivatives of the geometric invariants, which will 
give us, as particular cases, the shape derivatives of the Gaussian and mean 
curvatures.  The goal is to obtain them in terms of the geometric invariants.

We start with $i_1 = \kappa$:
   \begin{equation*}
   i_1'(\Gamma, \bs{v}) = I_1'(\Gamma, \bs{v}) = -\Dg^2v_n : 
\Dg\bs{n}^0-v_n 
I_{2} = - \lb v_n -v_n I_2,
   \end{equation*}
    which is consistent with the previous result (\ref{eq:curvature_shapeder}).
   Since $I_2=i_1^2-2i_2,$
   \begin{equation}\label{eq:firstinv}
   i_1'(\Gamma, \bs{v})=  - \lb v_n -v_n i_1^2(\Gamma) +2v_ni_2(\Gamma).
   \end{equation}
   
   For the second invariant, note that
   $I_2'(\Gamma, \bs{v}) = -2\left(\Dg^2v_n : \Dg\bs{n}+v_n I_{3} \right)$.
   Since $i_2=\frac{1}{2}(I_1^2-I_2)$, we have
   \begin{align*}
  i_2'(\Gamma, \bs{V}) &= I_1 I_1'(\Gamma, \bs{V})-\frac{1}{2}I_2'(\Gamma, 
\bs{V}) = -I_1\lb v_n-v_n I_1 I_2+ \Dg^2v_n : \Dg\bs{n}+v_n I_{3}\\
  &= -I_1 \lb v_n + \Dg^2v_n : \Dg\bs{n} + v_n(I_3-I_1I_2).
   \end{align*}
   
   To obtain a formula only involving the invariants $i_k$, observe that 
$i_3\definedas\frac{1}{3!} (I_1^3- 3 I_1 I_2 + 2 I_3)=\frac{1}{3}(I_3-I_1I_2 +i_1 i_2)$, as can be checked by replacing $i_1$ by 
$I_1$ and $i_2$ by $\frac12 (I_1^2-I_2)$ in the right-hand side, whence
   \begin{equation}\label{eq:secondinv}
   i_2'(\Gamma, \bs{V})=-i_1 \lb v_n +\Dg^2v_n : \Dg\bs{n} +v_n(3i_3-i_1i_2).
   \end{equation}
   
    If $N=3$, the Gaussian curvature $\kappa_g$ is the second 
invariant 
$i_2(\Gamma)$. Then, on the one hand, from (\ref{eq:firstinv}), we have the following expression for the 
shape derivative of the mean curvature $\kappa$ in terms of $\kappa_g$:
    \begin{equation}\label{eq:shape-der-k-kg}
    \kappa'(\Gamma, \bs{v})=  - \lb v_n -v_n \kappa^2 +2v_n \kappa_g.    
    \end{equation}
    
     On the other hand, since $i_3=0$ for $N=3$, we obtain from (\ref{eq:secondinv}) the 
following formula for the shape derivative of the Gaussian curvature.
    
    \begin{thrm}[Shape derivative of the Gauss curvature]
    For a $\mathcal{C}^2$-surface $\Gamma$ in $\Rn[3]$, the shape derivative 
of the Gaussian curvature $\kappa_g$ is given by
    \begin{equation*}
    \kappa_g'(\Gamma, \bs{V})=-\kappa \lb v_n +\Dg^2v_n : \Dg\bs{n} -v_n\kappa 
\kappa_g,
    \end{equation*}
    where $\kappa$ is the mean or additive curvature, $\bs{n}$ the normal 
vector 
field and $v_n=\bs{V}(0)\cdot \bs{n}$.
    \end{thrm}

\section{Application: A Newton-type method}\label{s:newton}
Most of shape optimization problems consist in finding a minimum of some 
functional restricted to a family of admissible sets (domains or surfaces), e.g., 
\begin{equation}\label{prob:min}
\Gamma_*=\argmin_{\Gamma\in \Ad} J(\Gamma) 
\end{equation}
For example, given a regular curve $\gamma$ in $\Rn[3]$, a minimal surface 
$\Gamma$ with boundary $\gamma$ is a solution of (\ref{prob:min})
with $J(\Gamma) =\int_{\Gamma} d\Gamma$, the area functional, and the admissible 
family $\Ad=\Ad(\gamma)$ consisting of all regular $2$ dimensional surfaces in 
$\Rn[3]$ with boundary $\gamma$. 

If $J$ is shape differentiable in $\Ad$, and $\Gamma_*$ is a minimizer, then $dJ(\Gamma_*, \bs{v})=0$ for all $\bs{v} \in \bs{\Vel}$, where $\bs{\Vel}$ is a vector space of admissible autonomous velocities, for 
example $\bs{\Vel} =\mathcal{C}^k_c(\D, \Rn[N])$.

We thus focus our attention in the following alternative problem:
\begin{equation}\label{prob:critic}
\text{Find } \Gamma_*\in \Ad :\qquad dJ(\Gamma_*, \bs{v})= 0,\quad  \text{ for all 
} 
\bs{v}\in\bs{\Vel}:= \mathcal{C}^k_c(\D, \Rn[N]).
\end{equation}

An interesting scheme to approximate the solutions of~\eqref{prob:critic} for 
surfaces of prescribed constant mean curvature was presented in~\cite{CMP2016}. 
There, results from numerical experiments document its performance and fast convergence.
The scheme was a variation of the Newton algorithm, which needs the computation of
second derivatives of the shape functional. The computations there were tailored to the specific problem of prescribed mean curvature, and based on variational calculus using parametrizations, rather than using shape calculus.

We first observe that, due to the structure Theorem (Theorem~\ref{t:structure}), Problem \eqref{prob:critic} is equivalent
to the following:
\begin{equation}\label{prob:critic_normal0}
\text{Find } \Gamma_*\in \Ad :\qquad dJ(\Gamma_*, v\nabla b_{\Gamma_*})= 0,\quad  \text{ for 
all } 
v\in\Vel_* \definedas \left\{ w \in \Vel : \left.\frac{\partial w}{\partial\n_*}\right|_{\Gamma_*} = 0 \right\},
\end{equation}
where $b_{\Gamma_*}\definedas b(\Omega_*)$ is the signed distance function corresponding to the domain $\Omega_*$ whose boundary is $\Gamma_*$, $\Vel = \mathcal{C}^k_c(\D)$ and $n_* = \nabla b_{\Gamma_*}$ is the normal vector to $\Gamma_*$.

We now present a scheme to approximate the solution of~\eqref{prob:critic_normal0} 
using a Newton-type method that generalizes the idea of~\cite{CMP2016}
in at least two ways. First, it uses the language of shape derivatives and secondly, it 
has the potential to work for a large class of shape functionals, not just the area functional.

We start by defining, for each $\Gamma \in \Ad$ and $v\in \Vel$, the functional $J_v(\Gamma) = dJ(\Gamma, v \nabla b_\Gamma)$, so that the solution $\Gamma_*$ satisfies $J_v(\Gamma_*) = 0$ for all $v \in \Vel_*$.
Assume now that $\Gamma_0\in\Ad$ is sufficiently close to the solution $\Gamma_*$ so that there exists $\bs{u}\in\bs{\Vel}$ (small, in some sense) such that $\Gamma_* := \Gamma_0+\bs{u}$, in the sense of Remark~\ref{r:taylor}; this Remark also implies that
 \begin{equation}
 J_v (\Gamma_*) = J_v(\Gamma_0 + \bs{u}) = J_v (\Gamma_0) +  dJ_v(\Gamma_0,\bs{u}) + o(|\bs{u}|).
 \end{equation}
The goal of finding $\Gamma_* = \Gamma_0 + \bs{u}$ such that $J_v(\Gamma_0 + \bs{u})=0$ is now switched to a simplified problem of finding $\bs{u}_0$ such that the linear approximation of $J_v$ around $\Gamma_0$ vanishes at $\Gamma_1 \definedas \Gamma_0 + \bs{u_0}$, i.e., $J_v (\Gamma_0) +  dJ_v(\Gamma_0,\bs{u}_0) = 0$. Another simplification arises when asking this equality to hold for all $v \in \Vel_0 \definedas \left\{ w \in \Vel : \frac{\partial w}{\partial\n}|_{\Gamma_0} = 0 \right\}$ (instead of $\Vel_*$ or $\Vel_1$). 

Since $dJ_v(\Gamma_0,\bs{u}_0)$ only depends on the normal component of $\bs{u}_0$ on $\Gamma_0$, this last problem has multiple solutions, so we restrict it by considering normal velocities of the form $\bs{u}_0 = u_0 \n_0$ with $u_0 \in V(\Gamma_0) = C^k(\Gamma_0)$ (and $\n_0$ the normal vector to $\Gamma_0$), and arrive at the following problem:
\begin{equation}
\text{Find } u_0 \in V(\Gamma_0) : \quad  J_v (\Gamma_0) +  dJ_v(\Gamma_0,u_0\n_0)=0 \quad \forall 
v\in \Vel_0. 
\end{equation}
Finally, define $\Gamma_1 = \Gamma_0 + u_0 \n_0$. This sets the basis for an iterative method that will be implemented and further investigated in forthcoming articles.

 \subsection{Area and Willmore functionals}
 
 We end this paper with the precise form of the second order shape derivatives of two important functionals: the area functional and the Willmore functional.
 
 For a shape differentiable boundary functional $J(\Gamma)=\int_{\Gamma} z(\Gamma)$, and a function $v \in C_c^k(\D)$, the functional $J_v(\Gamma) := dJ(\Gamma, v \nabla b_\Gamma)$ is given by $J_v (\Gamma)= \int_{\Gamma} z_v(\Gamma)+\kappa z(\Gamma) v$, where $z_v(\Gamma)=z'(\Gamma;v \nabla b_\Gamma)$. 
 Hence~(\ref{eq:integral-boundary-shape-der2}) yields 
 \[
dJ_v (\Gamma,\bs{u}) = 
 \int_{\Gamma} z_v'(\Gamma,\bs{u})+
 \big(\kappa'(\Gamma,\bs{u}) z(\Gamma) + \kappa(\Gamma) z'(\Gamma,\bs{u})\big)v
 + \kappa(\Gamma) z(\Gamma)v'(\Gamma,\bs{u})
 + \big(z_v(\Gamma)+ \kappa(\Gamma) z(\Gamma) v\big) \kappa(\Gamma) u,
 \]
 where $u=\bs{u}\cdot \nabla b_\Gamma$. 
 
 Since $v$ does not depend on $\Gamma$, Definitions~\ref{def:material-der} and \ref{def:shape-der-boundary-scalar} yield $v'(\Gamma, \bs{u})=\dot v(\Gamma,\bs{u}) - \gradg v\cdot\bs{u} = \nabla v \cdot \bs{u} - \gradg v \cdot \bs{u} = \frac{\partial v}{\partial n}u$. 
 Recall from~(\ref{eq:curvature_shapeder}) that $\kappa'(\Gamma, u)=-\lb u -u |\Dg\bs{n}|^2$. 
 Using the second invariant $i_2(\Gamma)=\frac{1}{2}(\kappa^2 - |D_\Gamma \n|^2)$, we can write
 \begin{equation}\label{eq:integral-formula}
dJ_v (\Gamma,\bs{u}) 
=\int_{\Gamma} 2i_2(\Gamma) z(\Gamma)\,uv +v \left(z_u(\Gamma)\kappa(\Gamma)-\lb u\, z(\Gamma)\right)+
u\,\kappa(\Gamma)\left(\frac{\partial v}{\partial n} z(\Gamma) +z_v(\Gamma)\right) + z_v'(\Gamma, \bs{u}).
 \end{equation}

We now apply formula \eqref{eq:integral-formula} to two examples of boundary 
functionals which make use of the results of previous sections to obtain 
useful formulas for $J_v(\Gamma)$ and $dJ_v(\Gamma, u)$.
\subsubsection{Area Functional}
For the area functional $J(\Gamma)= \int_{\Gamma} d\Gamma$, we have 
$z(\Gamma)\equiv 1$, $z_v(\Gamma) \equiv 0$ and $z_v'(\Gamma,u) \equiv 0$. Then 
$J_v(\Gamma)=\int_\Gamma \kappa(\Gamma) v$ and by (\ref{eq:integral-formula})
\[
dJ_v(\Gamma, \bs{u})=\int_\Gamma 2i_2(\Gamma) u v +\gradg v\cdot \gradg u
+ u\frac{\partial v}{\partial n}\kappa(\Gamma),
\]
where we have used an integration by parts formula, to replace $\int_\Gamma -\lb u \, v $ by $\int_\Gamma \gradg u \cdot \gradg v$.

\subsubsection{Willmore functional}
For the Willmore functional 
$J(\Gamma)=\int_{\Gamma}\frac{1}{2}\kappa(\Gamma)^2$ we have 
$z(\Gamma)=\frac{1}{2}\kappa(\Gamma)^2$ and by the product rule for shape derivatives (Remark~\ref{R:product shape derivatives})
$z_v(\Gamma)= z'(\Gamma,v\nabla b_\Gamma)=\kappa(\Gamma) \kappa'(\Gamma,v\nabla b_\Gamma) = -\kappa(\Gamma) \big(\lb v +v I_2(\Gamma)\big)$.
In order to apply formula~(\ref{eq:integral-formula}) we need to compute 
\[
z_v'(\Gamma,\bs{u}) = -\kappa'(\Gamma,\bs{u})\big(\lb v +v I_2(\Gamma)\big)
- \kappa(\Gamma) \big((\lb v)'(\Gamma,\bs{u}) +v'(\Gamma,\bs{u}) I_2(\Gamma)+v I_2'(\Gamma,\bs{u})\big) . 
\]
Recall that $\kappa'(\Gamma,\bs{u})=-\lb u -u |\Dg\bs{n}|^2$,  
$I_2'(\Gamma,\bs{u})=-2\left(\Dg^2u : \Dg\bs{n}+u I_3(\Gamma)\right)$ by Proposition~\ref{prop:derIp}, 
$v'(\Gamma,\bs{u})=u\frac{\partial v}{\partial n}$, and that
the shape derivative of $\lb v$ is, by Theorem 
\ref{thm:shape-der-lb},
\[
(\lb v)'(\Gamma, u)=\lb (u\frac{\partial v}{\partial n}) 
-u\left(2\Dg\bs{n}:\Dg^2v +\gradg \kappa \cdot \gradg v 
\right)
+  \gradg u \cdot \left( \kappa \gradg v - 2\Dg\bs{n}\gradg 
v\right).
\]
Putting all these ingredients together we obtain
\begin{align*}
dJ_v(\Gamma,\bs{u}) = {}& \int_\Gamma i_2(\Gamma)\kappa(\Gamma)^2 \, u\, v
- \kappa(\Gamma)^2 \, \lb u \, v 
- 2\kappa(\Gamma)^2 |D_\Gamma\n|^2 u \, v 
- \frac{\kappa(\Gamma)^2}2 \, \lb u \, v\\
&+ \frac{\kappa(\Gamma)^3}{2} \, u \frac{\partial v}{\partial n} - \kappa(\Gamma)^2 \, u \, \lb v 
+ \big( \lb u + | D_\Gamma \n|^2 u\big) \big( \lb v + | D_\Gamma \n|^2 v\big)\\
&-\kappa(\Gamma) \lb(u \frac{\partial v}{\partial n}) + 2\kappa(\Gamma) \, u D_\Gamma^2v:D_\Gamma \n  + \kappa(\Gamma) u \gradg v \cdot \gradg \kappa(\Gamma) - \kappa(\Gamma)^2 \gradg u \cdot \gradg v \\ &+ 2 \kappa(\Gamma) \gradg u \cdot D_\Gamma \n \gradg v 
- \kappa(\Gamma) | D_\Gamma \n|^2 u \frac{\partial v}{\partial n} 
+ 2 \kappa(\Gamma) v \, D_\Gamma^2u : D_\Gamma \n + 2 \kappa(\Gamma)I_3(\Gamma) \, u \, v.
\end{align*}
    
    \bibliographystyle{plain} 
    \bibliography{biblio}
    
\end{document}